\documentclass[oneside]{amsart}
\usepackage{enumerate,amssymb,mathtools}
\usepackage{mathrsfs}
\usepackage{graphicx}
\newtheorem{thm}{Theorem}[section]
\newtheorem{coro}[thm]{Corollary}
\newtheorem{prop}[thm]{Proposition}

\newtheorem*{claim}{Claim}

\newtheorem{lem}[thm]{Lemma}
\theoremstyle{definition}
\newtheorem{defn}[thm]{Definition}
\newtheorem{ex}[thm]{Example}

\newtheorem{question}[thm]{Question}

\newcommand{\Rset}{\mathbb{R}}

\newcommand{\Nset}{\mathbb{N}}

\newcommand{\Zset}{\mathbb{Z}}

\newcommand{\abs}[1]{\lvert#1\rvert}
\newcommand{\norm}[1]{\lVert#1\rVert}
\newcommand{\seq}[1]{\langle#1\rangle}

\newcommand{\eps}{\varepsilon}
\newcommand{\subs}{\subseteq}
\renewcommand{\leq}{\leqslant}
\renewcommand{\geq}{\geqslant}
\newcommand{\mc}[1]{\mathcal{#1}}

\DeclareMathOperator{\inter}{int}
\newcommand{\cl}[1]{\overline{#1}}

\DeclareMathOperator{\hdim}{\dim_{\mathsf{H}}}

\DeclareMathOperator{\diam}{diam}
\DeclareMathOperator{\dist}{dist}
\DeclareMathOperator{\DD}{\mathcal{D}}
\DeclareMathOperator{\DDapp}{\mathcal{D}_{\mathsf{app}}}
\newcommand{\simon}{\mbox{\sf\bfseries M\hspace{-0.15ex}o\hspace{-0.10ex}n}}

\DeclareMathOperator{\KK}{\mathcal{K}}
\DeclareMathOperator{\KKapp}{\mathcal{K}_{\mathsf{app}}}
\DeclareMathOperator{\MM}{\mathcal{M}}
\newenvironment{enum}{\begin{enumerate}[\rm(i)]}{\end{enumerate}}
\newenvironment{itemyze}%
  {\begin{list}{\textbullet}{\labelwidth1ex\setlength{\leftmargin}{2.1em}}}%
  {\end{list}}

\newcommand{\rest}{{\restriction}}

\newcommand{\si}{$\sigma$\nobreakdash-}

\newcommand{\gr}{\mathfrak{f}}
\newcommand{\ggr}{\mathfrak{g}}
\newcommand{\len}{\mathscr H^1}
\newcommand{\leb}{\mathscr L}

\newcommand{\urd}[1]{\overline#1{}^+}
\newcommand{\lrd}[1]{\underline#1{}^+}
\newcommand{\uld}[1]{\overline#1{}^-}
\newcommand{\lld}[1]{\underline#1{}^-}
\newcommand{\urdapp}[1]{\overline#1{}^+_{\mathsf{app}}}
\newcommand{\lrdapp}[1]{\underline#1{}^+_{\mathsf{app}}}
\newcommand{\uldapp}[1]{\overline#1{}^-_{\mathsf{app}}}
\newcommand{\lldapp}[1]{\underline#1{}^-_{\mathsf{app}}}

\newcommand{\dapp}[1]{#1'_{\mathsf{app}}}
\newcommand{\rd}[1]{#1^+}
\newcommand{\ld}[1]{#1^-}
\begin{document}
% -------------------------------------------------------------
\title
%[Fractal properties of monotone sets and graphs]
{Properties of functions with monotone graphs}

\author{Ond\v rej Zindulka}
\address
{Ond\v rej Zindulka\\
Department of Mathematics\\
Faculty of Civil Engineering\\
Czech Technical University\\
Th\'akurova 7\\
16000 Prague 6\\
Czech Republic}
\email{zindulka@mat.fsv.cvut.cz}
\urladdr{http://mat.fsv.cvut.cz/zindulka}
\author{Michael Hru\v s\'ak}
\address{Michael Hru\v s\'ak\\
Instituto de Matem\'aticas, UNAM, Apartado Postal 61-3,
Xangari, 58089, Morelia, Michoac\'an, M\'exico.}
\email{michael@matmor.unam.mx}
\author{Tam\'as M\'atrai}
\address
{Tam\'as M\'atrai\\
Alfr\' ed R\'enyi Institute of Mathematics\\
Hungarian Academy of Sciences\\
H-1053 Budapest\\
Re\' altanoda u.~13--15\\
Hungary}
\email{matrait@renyi.hu}
\author{Ale\v s Nekvinda}
\address
{Ale\v s Nekvinda\\
Department of Mathematics\\
Faculty of Civil Engineering\\
Czech Technical University\\
Th\'akurova 7\\
16000 Prague 6\\
Czech Republic}
\email{nales@mat.fsv.cvut.cz}
\author{V\'aclav Vlas\'ak}
\address
{V\'aclav Vlas\'ak\\
Department of Mathematical Analysis\\
Faculty of Mathematics and Physics\\
Charles University\\
Sokolovsk\' a 83\\
18675 Prague 8\\
Czech Republic}
\email{vlasakmm@volny.cz}
\thanks{%
Some of the work on this project was conducted during O.~Zindulka's sabbatical
stay at the Instituto de matem\'aticas, Unidad Morelia,
Universidad Nacional Auton\'oma de M\'exico supported by \mbox{CONACyT} grant no.~125108.
M. Hru\v s\'ak gratefully acknowledges support from PAPIIT grant IN101608 and
CONACYT grant 80355.
A.~Nekvinda was supported by MSM 6840770010 and the grant 201/08/0383 of
the Grant Agency of the Czech Republic.
V.~Vlas\'ak was supported by the grant 22308/B-MAT/MFF of the Grant Agency of the Charles
University in Prague and by grant 201/09/0067 of
the Grant Agency of the Czech Republic%
}

%%% ----------------------------------------------------------

\begin{abstract}
A metric space $(X,d)$ is \emph{monotone} if there is a linear order $<$ on $X$ and
a constant $c>0$ such that $d(x,y)\leq c d(x,z)$ for all $x<y<z\in X$.
Properties of continuous functions with monotone graph (considered as a planar set)
are investigated. It is shown, e.g., that such a function can be almost nowhere
differentiable, but must be differentiable at a dense set, and
that Hausdorff dimension of the graph of such a function is $1$.
\end{abstract}

%%% -----------------------------------------------------------
\maketitle
%%% -----------------------------------------------------------

%%%%%%%%%%%%%%%%%%%%%%%%%%%%%%%%%%%%%%%%%%%%%%%%%%%%%%%%%%%%%%%%%%%%%%%%%%%%%%%%
%%%%%%%%%%%%%%%%%%%%%%%%%%%%%%%%%%%%%%%%%%%%%%%%%%%%%%%%%%%%%%%%%%%%%%%%%%%%%%%%
\section{Introduction}
%%%%%%%%%%%%%%%%%%%%%%%%%%%%%%%%%%%%%%%%%%%%%%%%%%%%%%%%%%%%%%%%%%%%%%%%%%%%%%%%
%%%%%%%%%%%%%%%%%%%%%%%%%%%%%%%%%%%%%%%%%%%%%%%%%%%%%%%%%%%%%%%%%%%%%%%%%%%%%%%%

A metric space $(X,d)$ is called monotone if there is a linear order $<$ on $X$ and
a constant $c>0$ such that $d(x,y)\leq c d(x,z)$ for all $x<y<z\in X$.

Suppose $f$ is a continuous real-valued function defined on an interval.
The graph $\gr$ of $f$ is a subset of the plane. The goal of this paper is
to investigate differentiability of $f$ assuming that the graph $\gr$
is a monotone space.

%%%%%%%%%%%%%%%%%%%%%%%%%%%%%%%%%%%%%%%%%%%%%%%%%%%%%%%%%%%%%%%%%%%%%%%%%%%%%%%%
\subsection*{Monotone metric spaces}
%%%%%%%%%%%%%%%%%%%%%%%%%%%%%%%%%%%%%%%%%%%%%%%%%%%%%%%%%%%%%%%%%%%%%%%%%%%%%%%%
Monotone metric spaces were introduced in~\cite{Zin1,NZ2,MR2822419}.
Some applications are given in~\cite{Zin1,ZinQ,KMZ}.
\begin{defn}
Let $(X,d)$ be a metric space.

$(X,d)$ is called \emph{monotone} if there is a linear order $<$ on $X$ and
a constant $c>0$ such that for all $x,y,z\in X$
\begin{equation}\label{basic1}
  d(x,y)\leq c d(x,z) \qquad\text{whenever $x<y<z$}.
\end{equation}
The order $<$ is called a \emph{witnessing order} and $c$ is called
a \emph{witnessing constant}.

$(X,d)$ termed \emph{\si monotone} if it is a countable union of monotone subspaces.
\end{defn}
It is easy to check that if $(X,d)$, $c$ and $<$ satisfy \eqref{basic1}, then
$d(y,z)\leq(c+1)d(x,z)$ for all $x<y<z$.
It follows that replacing condition~\eqref{basic1} by
\begin{equation}\label{basic2}
  \max\bigl(d(x,y),d(y,z)\bigr)\leq c d(x,z) \qquad\text{whenever $x<y<z$}
\end{equation}
gives an equivalent definition of a monotone space.
Since we will be occasionally interested in the value of $c$, we introduce
the following notions.
\begin{defn}
Let $c>0$. A metric space $(X,d)$ is called
\begin{enum}
\item \emph{$c$-monotone} if there is a linear order $<$ such
that~\eqref{basic1} holds,
\item \emph{symmetrically $c$-monotone} if there is a linear order
$<$ such that~\eqref{basic2} holds.
\end{enum}
\end{defn}
It is clear that $(X,d)$ is monotone iff it is $c$-monotone for some $c$
iff it is symmetrically $c$-monotone for some $c$.
It is also clear that if a space is $c$-monotone, then it is
symmetrically $(c+1)$-monotone and that a symmetrically $c$-monotone space
is $c$-monotone.

Topological properties of monotone and \si mono\-tone spaces are investigated in
\cite{NZ2}. We recall the relevant facts proved therein.
A monotone metric space is suborderable, i.e.~embeds in a linearly orderable
metric space.
In particular, if $<$ is a witnessing order, then every open interval $(a,b)$
is open in the metric topology, i.e.~the metric topology is finer than
the order topology.

If a metric space contains a dense monotone subspace, then the space itself is
monotone. It follows that every \si monotone subset of a metric space
is contained in a \si monotone $F_\sigma$-subset. This fact will be utilized
at several occasions.

Though the topological dimension of a monotone metric space is at most one,
in a general context of a separable metric space there is nothing one can say
about the Hausdorff dimension of a monotone space. Indeed, there are $1$-monotone
compact spaces of arbitrary Hausdorff dimension, including $\infty$.
However, when one considers only monotone subspaces of Euclidean spaces,
there is, as proved in the oncoming paper~\cite{ZinFrac}, an upper estimate
of Hausdorff dimension by means of the witnessing constant.
On the other hand, by a result from~\cite{KMZ}, every Borel set in $\Rset^n$
contains a \si monotone subset of the same Hausdorff dimension.
Thus a monotone set can have Hausdorff dimension greater than $1$.
The same holds for curves: by an unpublished result of Pieter Allaart and
Ond\v rej Zindulka, the von Koch curve is monotone.

The interplay between porosity and
monotonicity in the plane are investigated in~\cite{HrZi,ZinFrac}.
In particular, by~\cite[Theorem 4.2]{HrZi}, every monotone set in $\Rset^n$ is
strongly porous (see Section~\ref{sec:MICHAL} for the definition).
This fact is utilized in Section~\ref{sec:MICHAL}.

%%%%%%%%%%%%%%%%%%%%%%%%%%%%%%%%%%%%%%%%%%%%%%%%%%%%%%%%%%%%%%%%%%%%%%%%%%%%%%%%
\subsection*{Monotone graphs}
%%%%%%%%%%%%%%%%%%%%%%%%%%%%%%%%%%%%%%%%%%%%%%%%%%%%%%%%%%%%%%%%%%%%%%%%%%%%%%%%

We will focus on properties of continuous functions that have monotone graph.
Our hope was that such a function must be differentiable at a substantial
portion of its domain. It, however, turned out that the interplay between
monotonicity of graph and differentiability is more delicate and definitely
not straightforward. Our goal is to study this interplay.

It turns out that such a graph has \si finite $1$-dimensional Hausdorff measure
and in particular, in contrast with the just mentioned von Koch curve property,
has Hausdorff dimension $1$.
Can one go further and prove for instance that a continuous function with
a monotone graph is differentiable at a large set, say, almost everywhere?
Or, in the other direction, that a differentiable function has a monotone or
\si monotone graph?
We provide answers to these questions.

We outline a few results.
In sections~\ref{sec:ONDREJ} and~\ref{sec:ondrej2} we show,
e.g., that a differentiable function has a \si monotone graph and that
a continuous function with monotone graph has knot points
(i.e.~both upper/lower Dini derivatives are $\infty$/$-\infty$)
almost everywhere where it does not posses a derivative.

In section~\ref{sec:m1} continuous functions with a $1$-monotone graph are
investigated. The strongest result says that such a function on a compact
interval is of finite variation and in particular is differentiable almost
everywhere.

However, in Section~\ref{sec:VASEK} we construct a continuous function that
exhibits that, perhaps surprisingly, this theorem completely fails for
monotone graph:
an almost nowhere differentiable function with a monotone graph. Consequently,
almost all points of the domain are knot points.

So a continuous function with monotone graph can be rather wild.
But not completely: in Section~\ref{sec:ondrej2} we show that such a function
is differentiable at an uncountable dense set and its graph is of \si finite
length and in particular of Hausdorff dimension $1$.

As proved at the beginning of Section~\ref{sec:MICHAL}, a graph of
an absolutely continuous function is \si monotone except a set
of linear measure zero. The following result
is thus perhaps surprising: there is an absolutely continuous function
whose graph is not \si monotone. Moreover, such a function can be constructed
so that the graph is a porous set.

The concluding Section~\ref{sec:Q} lists some open problems.

%%%%%%%%%%%%%%%%%%%%%%%%%%%%%%%%%%%%%%%%%%%%%%%%%%%%%%%%%%%%%%%%%%%%%%%%%%%%%%%%
%%%%%%%%%%%%%%%%%%%%%%%%%%%%%%%%%%%%%%%%%%%%%%%%%%%%%%%%%%%%%%%%%%%%%%%%%%%%%%%%
\section{Monotone graphs}\label{sec:PRE}
%%%%%%%%%%%%%%%%%%%%%%%%%%%%%%%%%%%%%%%%%%%%%%%%%%%%%%%%%%%%%%%%%%%%%%%%%%%%%%%%
%%%%%%%%%%%%%%%%%%%%%%%%%%%%%%%%%%%%%%%%%%%%%%%%%%%%%%%%%%%%%%%%%%%%%%%%%%%%%%%%

A topological closure and interior of a set $A$ in a metric space are denoted
by $\cl A$ and $\inter A$, respectively.

For $A\subs\Rset^2$ denote $\hdim A$ the Hausdorff dimension of $A$.
Lebesgue measure on the line is denoted by $\leb$.
Given $A\subs\Rset^2$, its linear measure, i.e.~$1$-dimensional Hausdorff measure,
is denoted $\len(A)$ and referred to as \emph{Hausdorff length}.

We will be concerned with monotonicity of graphs of continuous functions.
The symbol $I$ is used to denote a non-degenerate interval of real numbers.
Let $f:I\to\Rset$ be a continuous function.
Formally there is no difference between $f$ and its graph,
but confusion may arise for instance from ``$f$ is monotone''.
Therefore we use $\gr$ when referring to the graph of $f$
as a pointset in the plane (and likewise $\ggr$ for the graph of $g$ etc.).
Given a set $E\subs I$, denote $\gr|E$ the graph of $f$ restricted to $E$.

We write $\psi_f(x)=(x,f(x))$
(or just $\psi(x)$ if there is no danger of confusion) to denote the
natural parametrization of $\gr$. The graph $\gr$ is obviously a connected
linearly ordered space. By~\cite[Theorem II]{MR0003201}, if a space is
linearly orderable and connected,
then the order is unique up to reversing. Therefore there are only two
orders on $\gr$ that can witness monotonicity of $C$: the order given by
$\psi(t)<\psi(s)$ if $t<s$ and its reverse.
Since being symmetrically $c$-monotone is invariant with respect to reversing
the witnessing order, it does not matter which of the two orders we choose.
Overall, given the conditions
\begin{align}
&\text{for all $x<y<z\in I$}\quad|\psi(x)-\psi(y)|\leq c|\psi(x)-\psi(z)|,
  \label{L1,1}\\
&\text{for all $x<y<z\in I$}\quad|\psi(z)-\psi(y)|\leq c|\psi(x)-\psi(z)|,
  \label{L1,2}
\end{align}
we have
\begin{lem} If $f:I\to\Rset$ is continuous, then
\begin{enum}
\item $\gr$ is $c$-monotone if and only if at least one of
\eqref{L1,1}, \eqref{L1,2} holds,
\item $\gr$ is symmetrically $c$-monotone if and only if both
\eqref{L1,1} and \eqref{L1,2} hold.
\end{enum}
\end{lem}
The following simple condition equivalent to monotonicity of $\gr$ will
turn useful.
\begin{defn}
Given $c>0$, say that $f$ satisfies condition $\mathsf P_c$ if
\begin{equation}\label{Pc}
  \max_{x\leq t\leq y}\abs{f(x)-f(t)}\leq c|x-y|
  \text{ whenever $x<y$ and $f(x)=f(y)$}.
  \tag{$\mathsf P_c$}
\end{equation}
\end{defn}
\begin{lem}\label{L4}
Let $f:I\to\Rset$ be a continuous function and $c\geq1$.
\begin{enum}
\item If $\gr$ is $c$-monotone, then $f$ satisfies $\mathsf P_c$,
\item if $f$ satisfies $\mathsf P_{c-1}$, then $\gr$ is symmetrically $c$-monotone.
\end{enum}
\end{lem}
\begin{proof}
(i)
Let $x<y$ satisfy $f(x)=f(y)$. $c$-monotonicity of $\gr$ yields for all $t\in[x,y]$
$$
  \abs{f(x)-f(t)}\leq\abs{\psi(x)-\psi(t)}\leq c\abs{\psi(x)-\psi(y)}=c\abs{x-y}.
$$

(ii)
We prove only condition~\eqref{L1,1}, as
condition~\eqref{L1,2} is proved in the same manner.
Let $x<y<z\in I$. Suppose that $f(x)\leq f(z)\leq f(y)$, all other cases are
trivial or similar.
Find $w\in [x,y]$ such that $f(w)=f(z)$.
Condition $\mathsf P_{c-1}$ yields $\abs{f(z)-f(y)}\leq (c-1)|z-w|$. Therefore
\begin{align*}
  \abs{\psi(x)-\psi(y)}&\leq\abs{(x-z,f(x)-f(y))}\\
  &\leq\abs{\psi(x)-\psi(z)}+\abs{(z-z,f(z)-f(y))}\\
  &\leq\abs{\psi(x)-\psi(z)}+(c-1)\abs{z-w}
  \leq c\abs{\psi(x)-\psi(z)}.
  \qedhere
\end{align*}
\end{proof}

Monotonicity and \si monotonicity are clearly global properties.
It turns out that the following pointwise counterpart of monotonicity
is worth investigation.
\begin{defn}
Let $f:I\to\Rset$ be a continuous function and $c\geq1$.

\textbullet{}
Say that $\gr$ is \emph{$c$-monotone at $y\in\gr$} if there is a neighborhood
$U\subs\gr$ of $y$ such that if $x<y<z$ and $x,z\in U$, then
$\abs{x-y}\leq c\abs{x-z}$, and \emph{monotone at $y$} if it is $c$-monotone at
$y$ for some $c\geq1$.
The set of all points where $\gr$ is $c$-monotone is denoted $\simon_c(\gr)$.
The set of all points where $\gr$ is monotone is denoted $\simon(\gr)$.

\textbullet{}
If $\gr$ is $c$-monotone (monotone) at $\psi_f(y)$, we call $y$ an
\emph{$\MM_c$-point} (\emph{$\MM$-point}) of $f$.
The set of all $\MM_c$-points of $f$ is denoted $\MM_c(f)$ or just $\MM_c$.
The set of all $\MM$-points of $f$ is denoted $\MM(f)$ or just $\MM$.
\end{defn}
It is clear that $\simon(\gr)=\gr|\MM(f)$.
Since the natural parametrization $\psi_f$ is a homeomorphism, it thus makes
no difference whether we investigate topological properties of $\simon(\gr)$
or $\MM(f)$.

Obviously, $y\in I$ is an $\MM_c$-point if and only if there is $\eps>0$
such that
\begin{equation}\label{pwm}
  \text{for all $x\in(y-\eps)$, $z\in(y+\eps)$}
  \quad\abs{\psi(x)-\psi(y)}\leq c\abs{\psi(x)-\psi(z)}.
\end{equation}
By the reasoning preceding~\eqref{basic2}, the inequality
$\abs{\psi(x)-\psi(y)}\leq c\abs{\psi(x)-\psi(z)}$ in condition~\eqref{pwm}
can be replaced with $\abs{\psi(y)-\psi(z)}\leq c'\abs{\psi(x)-\psi(z)}$,
possibly with another constant $c'$.
Thus the definition of $\MM$-point is ``symmetric'', in that it is invariant
under reversing the orientation of $x$- or $y$-axis.

Another equivalent definition: $y$ is an $\MM$-point if and only if there
is $c$ and $\eps>0$ such that
\begin{equation}\label{pwm2}
  \text{for all $x\in(y-\eps)$, $z\in(y+\eps)$}
  \quad\abs{f(x)-f(y)}\leq c\bigl(\abs{f(x)-f(z)}+\abs{z-x}\bigr).
\end{equation}

Let us clarify the relation of monotonicity, \si monotonicity,
pointwise monotonicity and $\MM$-points.
The proof of the following is straightforward.
\begin{prop}\label{meager2}
If $f:I\to\Rset$ is continuous, then
\begin{enum}
\item $\MM$ and $\MM_c$ are $F_\sigma$-sets, and so are $\simon(\gr)$ and
$\simon_c(\gr)$,
\item $\simon(\gr)=\gr|\MM$ is \si monotone,
\item $\simon_c(\gr)=\gr|\MM_c$ is a countable union of closed $c$-monotone sets.
\end{enum}
\end{prop}
Needles to say that if a continuous function has a monotone graph,
then all points are $\MM$-points. However, there is a continuous function
$f$ on $[0,1]$ with $\MM_1(f)=[0,1]$, i.e.~$\gr$ is $1$-monotone at each point,
but $\gr$ is not monotone:
let $f(x)=\abs x^{3/2}\sin\frac1x$ for $x\neq0$, $f(0)=0$.
It is easy to check that condition~\eqref{Pc} fails for
each $c$ and thus $\gr$ is not monotone. On the other hand, $f$
is differentiable everywhere and hence, by Theorem~\ref{thmMM1} below,
all points are $\MM_1$-points.

By~\cite[Corollary 2.6]{NZ2}, every monotone set has a monotone closure.
Using this fact, the above proposition and Baire category theorem
one can easily prove the following facts on relation between monotonicity,
\si monotonicity and pointwise monotonicity.
\begin{lem}\label{baire}
If $f:I\to\Rset$ is continuous, with a \si monotone graph, then
for any interval $J\subs I$ there is
a subinterval $J'\subs I$ such that $\gr|J'$ is monotone.
\end{lem}
\begin{coro}\label{intM}
Let $f:I\to\Rset$ be continuous.
\begin{enum}
\item If $f$ has a \si monotone graph, then $\inter\MM$ is dense in $I$,
i.e.~$\simon(\gr)$ contains an open dense subset of $\gr$.
\item If all points of $I$ are $\MM$-points, i.e.~if $\gr$ is monotone
at each point, then $\gr$ is \si monotone.
\end{enum}
\end{coro}
Part (i) of this corollary cannot be strengthened:
As shown in~\ref{chdim}, there is a continuous function $f$ on $[0,1]$
with a \si monotone graph, and a perfect set of non-$\MM$-points.

There is a profound connection between the set of $\MM$-points and monotone
subspaces of $\gr$. Its proof is straightforward.

\begin{prop}\label{meager3}
If $f:I\to\Rset$ is continuous, then the following are equivalent.
\begin{enum}
\item Every monotone set $M\subs\gr$ is nowhere dense in $\gr$,
\item every monotone set $M\subs\gr$ is meager in $\gr$,
\item $\MM(f)$ is meager in $I$,
\item $\inter\MM(f)=\emptyset$.
\end{enum}
\end{prop}

%%%%%%%%%%%%%%%%%%%%%%%%%%%%%%%%%%%%%%%%%%%%%%%%%%%%%%%%%%%%%%%%%%%%%%%%%%%%%%%%
%%%%%%%%%%%%%%%%%%%%%%%%%%%%%%%%%%%%%%%%%%%%%%%%%%%%%%%%%%%%%%%%%%%%%%%%%%%%%%%%
\section{Differentiability vs.~pointwise monotonicity}
\label{sec:ONDREJ}
%%%%%%%%%%%%%%%%%%%%%%%%%%%%%%%%%%%%%%%%%%%%%%%%%%%%%%%%%%%%%%%%%%%%%%%%%%%%%%%%
%%%%%%%%%%%%%%%%%%%%%%%%%%%%%%%%%%%%%%%%%%%%%%%%%%%%%%%%%%%%%%%%%%%%%%%%%%%%%%%%

We now investigate if pointwise monotonicity is related to differentiability.
Recall definitions of derivatives and related notation.
The \emph{upper right Dini derivative} of a function $f:I\to\Rset$ at point $x$
is denoted and defined by
$\urd f(x)=\limsup_{y\to x+}\frac{f(y)-f(x)}{y-x}$.
The other three Dini derivatives $\lrd f(x)$, $\uld f(x)$ and $\lld f(x)$
are defined likewise. If the four Dini
derivatives at $x$ equal, the common value is of course the derivative $f'(x)$.
If the two right Dini derivatives
at $x$ are equal, the common value is called the \emph{right
derivative} and denoted $\rd f(x)$; and likewise for the left side.
The set of points where the derivative of $f$ exists (infinite values are allowed)
is denoted $\DD(f)$ or just $\DD$.

A point $x\in I$ is called a \emph{knot point of $f$} if
$\uld{f}(x)=\urd{f}(x)=\infty$ and $\lld{f}(x)=\lrd{f}(x)=-\infty$.
The set of knot points of $f$ is denoted $\KK(f)$ or just $\KK$.

The \emph{approximate upper right Dini derivative} of $f$ at point
$x$ is denoted and defined by
$$
  \urdapp f(x)=
  \inf\Bigl\{t:\lim_{\delta\to0+}\frac1\delta\,
  \leb\bigl(\bigl\{y\in(x,x+\delta):\tfrac{f(y)-f(x)}{y-x}
  \leq t\bigr\}\bigr)=1\Bigr\}.
$$
The other three approximate Dini derivatives $\lrdapp f(x)$, $\uldapp f(x)$
and $\lldapp f(x)$ are defined likewise, as well as the
\emph{approximate derivative}
$\dapp f(x)$ and \emph{right} and \emph{left approximate derivatives}.
The set of points where the approximate derivative of $f$ exists
is denoted $\DDapp(f)$ or just $\DDapp$.
Approximate knot points are defined in the obvious way. The set of approximate knot
points of $f$ is denoted $\KKapp(f)$ or just $\KKapp$.
\begin{lem}\label{MMM}
Let $f:I\to\Rset$ be continuous and $y\in I$.
\begin{enum}
\item If $y$ is not an $\MM$-point, then $\urd f(y)=-\lld f(y)=\infty$,
or $\uld f(y)=-\lrd f(y)=\infty$.
\item If $\lld f(y)=-\infty$ and $\lrdapp f(y)=\infty$, then $y$ is not an
$\MM$-point.
\end{enum}
\end{lem}
\begin{proof}
(i) Let $y=f(y)=0$. Since $0$ is not an $\MM$-point, it follows
from~\eqref{pwm2} that there are sequences $x_n\nearrow 0$ and $z_n\searrow 0$
such that
\begin{equation}\label{densx2}
  \abs{f(x_n)}\geq n\bigl(\abs{f(x_n)-f(z_n)}+\abs{x_n-z_n}\bigr).
\end{equation}
In particular, $\abs{f(x_n)}\geq n\abs{x_n-z_n}\geq n\abs{x_n}$.
\emph{Mutatis mutandis} we may assume that all $f(x_n)$'s have the same sign.
Suppose that $f(x_n)>0$ for all $n$; the other case is treated likewise.
Hence $f(x_n)\geq n\abs{x_n}$. Clearly
$f(z_n)\geq f(x_n)-\abs{f(z_n)-f(x_n)}$, and~\eqref{densx2} yields
$\abs{f(z_n)-f(x_n)}\leq\frac1n f(x_n)$. Therefore $f(z_n)\geq f(x_n)(1-\frac1n)$.
Apply~\eqref{densx2} again to get $f(z_n)\geq n(1-\frac1n)z_n=(n-1)z_n$.
In summary, $f(x_n)\geq n\abs{x_n}$ and $f(z_n)\geq(n-1)z_n$, which is enough for
$\urd f(y)=-\lld f(y)=\infty$.

(ii) Suppose for contrary that $y$ is an $\MM$-point and assume without loss
of generality that $y=f(y)=0$. Let $\eps$ and $c$ be such that~\eqref{pwm2} holds.
Let $a=4c$ and $\beta=1/a$.

Since $\lrdapp f(0)>a$, there is $\delta<\eps$ such that for all $s<\delta$
\begin{equation}\label{densx}
  \leb\bigl\{t\in(0,s):\tfrac{f(t)}{t}>a\bigr\}>s(1-\beta).
\end{equation}
Since $\lld f(0)<-a$, there is $s\in(0,\delta)$ such that $\frac{f(-s)}{s}\geq a$.
Therefore there is $x\in[-s,0)$ such that $f(x)=as$.
Now use~\eqref{densx} to conclude that there is $t\in(s(1-\beta),s)$ such that
$\frac{f(t)}{t}>a$, i.e.~$f(t)>at$, and choose $z\in(0,t)$ such that
$f(z)=at$. Clearly $x\in(-\eps,0)$ and $z\in(0,\eps)$.
However, $\abs{f(x)-f(y)}=f(x)=as$,
$\abs{f(x)-f(z)}=as-at\leq as-as(1-\beta)=as\beta=s$ and $\abs{z-x}\leq2s$.
Thus~\eqref{pwm2} yields
$as\leq c(s+2s)$, which is contradicted by $a=4c$.
\end{proof}

\begin{thm}\label{thmMM}
If $f:I\to\Rset$ is continuous, then
\begin{enum}
\item $\DD(f)\subs\MM(f)$,
\item there is a set $E\subs I$ such that $\len(\gr|E)=0$ and
$\MM(f)\subs\DDapp(f)\cup\KKapp(f)\cup E$.
In particular, almost every $\MM$-point $x\notin\DDapp(f)$ is a knot point.
\end{enum}
\end{thm}
\begin{proof}
(i) It follows from Lemma~\ref{MMM}(i) that if $x$ is not an $\MM$-point,
then there are two Dini derivatives at $x$ that differ. Therefore $x\notin\DD(f)$.

(ii) We employ the approximate derivative version of the famous
Denjoy--Young--Saks Theorem due to Alberti, Csornyei, Laczkovich and
Preiss~\cite{MR1825530} that strengthens the Denjoy--Khintchine Theorem:

\emph{If $f$ is measurable, then there is a set $E\subs I$ such that
$\len(\gr|E)=0$ and for every point $x\notin E$ either $\dapp f(x)$
exists and is finite, or else all approximate Dini derivatives are infinite.}

It follows that if $x\notin\DDapp\cup\KKapp\cup E$, then all possible
configurations of the Dini derivatives obtain by reversing
the $x$- or $y$-axis from the following two cases:
\begin{itemize}
\item $\lldapp f(x)=\uldapp f(x)=-\infty$,  $\lrdapp f(x)=\urdapp f(x)=+\infty$,
\item $\lldapp f(x)=-\infty$, $\uldapp f(x)=\lrdapp f(x)=\urdapp f(x)=+\infty$.
\end{itemize}
Both satisfy the hypotheses of Lemma~\ref{MMM}(ii).
Hence $x$ is not an $\MM$-point.

The second statement of (ii) follows from the obvious inclusion $\KKapp\subs\KK$.
\end{proof}

The last goal of this section is to derive from Theorem~\ref{thmMM}(ii)
that the set of points where the graph is monotone has \si finite length
and in particular Hausdorff dimension $1$.
We need the following folklore covering lemma. Instead of reference we provide
a brief proof.
\begin{lem}\label{2rlemma}
Let $X$ be a metric space and $E\subs X$. Let $\{r_x:x\in E\}$ be a set of
positive reals
such that $\sup_{x\in E}r_x<\infty$. Then for each $\delta>2$ there is a set
$D\subs E$ such that the family $\{B(x,r_x):x\in D\}$ is disjoint and
the family $\{B(x,\delta r_x):x\in D\}$ covers $E$.
\end{lem}
\begin{proof}
We may assume that $r_x<1$ for all $x\in E$. Define recursively
\begin{align*}
  A_n&=\{x\in E:(\delta-1)^{-n+1}>r_x\geq (\delta-1)^{-n}\},\\
  B_n&=\{x\in A_n:B(x,r_x)\cap\bigcup\nolimits_{i<n}\bigcup\mathcal A_i=\emptyset\}
\end{align*}
and let $\mathcal A_n\subs\{B(x,r_x):x\in B_n\}$ be a maximal disjoint family.
It is easy to check that $D=\{x\in E:B(x,r_x)\in\bigcup_{n=0}^\infty\mathcal A_n\}$
is the required set.
\end{proof}

\begin{lem}\label{lemE}
Suppose that $f:I\to\Rset$ is continuous.
$$
  E^+=\{x\in[0,1]:\exists x_n\downarrow x\text{ such that $f(x_n)=f(x)$}\}.
$$
If $A\subs E^+\cap\MM_c(f)$, then $\len(\gr|A)\leq4c\leb(A)$.
\end{lem}
\begin{proof}
Let $A\subs E^+$.
For $\eps>0$ let $A_\eps$ be the set of points $y\in A$ satisfying~\eqref{pwm}.
Fix $\eta\in(0,\eps)$. Let $\{U_i:i\in\Nset\}$ be a cover of $A_\eps$
by open intervals of length $<\eta$ such that
$\sum_i\diam(U_i)<\leb(A_\eps)+\eta$.

Now fix $i\in\Nset$.
For each $x\in A_\eps\cap U_i$ choose $z_x>x$, $z_x\in U_i$ such that
$f(z_x)=f(x)$.
If $y\in[x,z_x]\cap A_\eps$, then, since $\abs{z_x-x}\leq\eta<\eps$, condition
~\eqref{pwm} with $z-z_x$ is met. Hence
$$
  \abs{\psi(y)-\psi(x)}\leq c\abs{\psi(z_x)-\psi(x)}=c\abs{z_x-x}.
$$
It follows that letting $r_x=c(z_x-x)$ we have
$$
  \gr\big|([x,z_x]\cap A_\eps)\subs B(\psi(x),r_x).
$$
The family $\mc B=\{B(\psi(x),r_x):x\in A_\eps\cap U_i\}$ thus covers
$A_\eps\cap U_i$.
Apply Lemma~\ref{2rlemma}: for any $\delta>2$
there is a set $A'\subs A_\eps\cap U_i$ such that the family
$\{B(\psi(x),r_x):x\in A'\}$ is pairwise disjoint and
$\gr|(A_\eps\cap U_i)\subs\bigcup_{x\in A'}B(\psi(x),\delta r_x)$.
We claim that the family of intervals $\{[x,z_x]:x\in A'\}$ is pairwise disjoint.
Indeed, if $x,y\in A'$ were such that $[x,z_x]\cap[y,z_y]\neq\emptyset$,
then either $x\in[y,z_y]\cap A_\eps$ or $y\in[x,z_x]\cap A_\eps$. Suppose
the former. Then $\psi(x)\in\psi([y,z_y]\cap A_\eps)\subs B(\psi(y),r_y)$.
Therefore the balls $B(\psi(x),r_x)$ and $B(\psi(y),r_{y})$ would not
be disjoint.

It follows that $\sum_{x\in A'}\abs{x-z_x}\leq\diam(U_i)$, which yields
$$
  \sum_{x\in A'}\diam(B(\psi(x),\delta r_x))
  \leq 2\delta\sum_{x\in A'}r_x
  \leq 2\delta c\sum_{x\in A'}\abs{x-z_x}
  \leq 2\delta c\diam(U_i).
$$
Moreover, the diameters of $B(\psi(x),\delta r_x)$ do
not exceed $2\delta c\eta$. Consequently
$$
  \len_{2\delta c\eta}(\gr|(A_\eps\cap U_i)\leq 2\delta c\diam(U_i).
$$
Summing over $i$ yields
$$
  \len_{2\delta c\eta}(\gr|A_\eps)\leq 2\delta c
  \sum_{i\in\Nset}\diam(U_i)
  \leq2\delta c(\leb(A_\eps)+\eta).
$$
$\len(\gr|A_\eps)\leq4c\leb(A_\eps)$ now follows on letting $\eta\to0$ and
$\delta\to2$, and $\len(\gr|A)\leq4c\leb(A)$ on letting $\eps\to0$.
\end{proof}

\begin{thm}\label{thmHaus}
If $f:I\to\Rset$ is continuous, then $\len(\simon(\gr))$ is \si finite.
In particular, $\hdim\simon(\gr)=1$.
\end{thm}
\begin{proof}
Let $A=\{\psi(x):x\notin\KK\}\subs\gr$. It is clear that for any point
$a\in A$ there is a (one-sided) cone $V$ with vertex at $a$ and a ball $B$ centered
at $a$ such that the only point of $\gr$ within $V\cap B$ is $a$.
Such a set is by~\cite[Lemma 15.13]{1333890} \emph{rectifiable}, i.e.~%
covered by countably many Lipschitz curves. In particular, $A$ has \si finite
length.
In view of Theorem~\ref{thmMM}(ii) it thus remains to show that
$\len(\gr|\MM\cap\KK)$ is \si finite. But that follows at once from
Lemma~\ref{lemE}, since any knot point belongs to the set $E^+$.
\end{proof}
\begin{coro}\label{coroHaus}
If $f:I\to\Rset$ is continuous with a monotone graph, then $\len(\gr)$
is \si finite. In particular, $\hdim\gr=1$.
\end{coro}

%%%%%%%%%%%%%%%%%%%%%%%%%%%%%%%%%%%%%%%%%%%%%%%%%%%%%%%%%%%%%%%%%%%%%%%%%%%%%%%%
%%%%%%%%%%%%%%%%%%%%%%%%%%%%%%%%%%%%%%%%%%%%%%%%%%%%%%%%%%%%%%%%%%%%%%%%%%%%%%%%
\section{Functions with a monotone or \si monotone graph}
\label{sec:ondrej2}
%%%%%%%%%%%%%%%%%%%%%%%%%%%%%%%%%%%%%%%%%%%%%%%%%%%%%%%%%%%%%%%%%%%%%%%%%%%%%%%%
%%%%%%%%%%%%%%%%%%%%%%%%%%%%%%%%%%%%%%%%%%%%%%%%%%%%%%%%%%%%%%%%%%%%%%%%%%%%%%%%

In this section we investigate differentiability of continuous functions
with monotone or \si monotone graph.

Our first theorem claims that if $\gr$ is monotone, then the
approximate derivatives coincide with derivatives.
\begin{prop}\label{approximate}
If $f:I\to\Rset$ is continuous with a monotone graph,
then $\urdapp f(x)=\urd f(x)$ for all $x\in I$.
A similar statement holds for all Dini derivatives.
\end{prop}
\begin{proof}
Assume for contrary that there is $x$ such that $\urdapp f(x)<\urd f(x)$.
\emph{Mutatis mutandis} we may suppose that $x=f(x)=0$.
Choosing suitable constants $\alpha,\beta$ the function
$g(y)=\alpha f(y)-\beta y$ satisfies
$\urdapp g(0)<0$ and $\urd g(0)>1$.
Since the graph of $g$ is an affine transform of the graph of $f$ and
an affine transform is bi-Lipschitz, the graph of $g$ is
by~\cite[Proposition 2.2]{NZ2} a monotone set.
Therefore there is $c\geq1$ such that $g$ satisfies condition~\ref{Pc}.

Since $\urdapp g(0)<0$, the set $M=\{y\in I:g(y)<0\}$ satisfies
\begin{equation}\label{eq:approx1}
  \forall\eps>0\ \exists\delta_0\ \forall\delta\in(0,\delta_0)
  \quad\leb\bigl([0,\delta]\setminus M\bigr)<\eps\delta
  %\wedge\forall y\in[0,\delta]\cap M\  g(y)<0.
\end{equation}
Let $\eps=\frac1{2c}$ and let $\delta_0$ satisfy \eqref{eq:approx1}.
Since $\urd g(0)>1$, there is $t\in(0,\delta_0/2)$ such that $g(t)>t$.
Put $\delta=2t$. Since $\eps\leq\frac12$ and $\delta<\delta_0$,
\eqref{eq:approx1} yields
$M\cap(0,t)\neq\emptyset$ and $M\cap(t,2t)\neq\emptyset$.
Therefore the numbers $a=\sup(0,t)\cap M$, $b=\inf(t,2t)\cap M$
satisfy $0<a<t<b<\delta$. Also $g(a)=g(b)=0$ by the continuity of $g$.
Obviously $[a,b]\cap M=\emptyset$. Hence \eqref{eq:approx1} yields
$\abs{b-a}<\eps\delta=t/c$.
Therefore
$
  c\abs{b-a}<t<g(t)=\abs{g(t)-g(a)}
$
and thus condition~\ref{Pc} fails: the desired contradiction.
\end{proof}
This theorem together with Theorem~\ref{thmMM} yield
\begin{coro}
If $f:I\to\Rset$ is continuous function with a monotone graph,
then there is a set $E\subs I$ such that $\len(\gr|E)=0$ and
$I=\DD(f)\cup\KK(f)\cup E$.
In particular, almost all points $x\notin\DD(f)$ are knot points.
\end{coro}
\begin{coro}
If $f:I\to\Rset$ is a continuous function with a $c$-monotone graph, then
$\len(\gr|A)\leq4c\leb(A)$ for every $A\subs I\setminus\DD(f)$.
In particular, $\len(\gr|I\setminus\DD(f))<\infty$.
\label{clp2}
\end{coro}

In the next section we present an example of function with a monotone graph
that has derivative almost nowhere, hence the set of knot points is rather large.
However, a \si monotone graph yields a dense set of differentiability.

\begin{lem}
Let $f:I\to\Rset$ be a continuous function with a \si monotone graph.
Then $\leb(f[\DD(f)\cap J])>0$ for each interval $J\subs I$
where $f$ is not constant.      \label{bigder}
\end{lem}
\begin{proof}
Using Lemma~\ref{baire}
it is clearly enough to prove that if $f:[0,1]\to\Rset$, $f(0)\neq f(1)$,
and $\gr$ is monotone, then $\leb(f[\DD])>0$.
Suppose the contrary: $\leb(f[\DD])=0$.
Let $0\leq x<y\leq1$. Use the assumption and Corollary~\ref{clp2} to estimate
$\abs{f(x)-f(y)}$:
\begin{align*}
  \abs{f(x)-f(y)}&=\leb([f(x),f(y)])\leq\leb(f[x,y])\\
  &\leq\leb(f[[x,y]\setminus\DD])+\leb(f[[x,y]\cap\DD])\\
  &\leq 4c\,\leb([x,y]\setminus\DD)+\leb(f[\DD])
  \leq 4c\,\leb([x,y])=4c\abs{x-y}.
\end{align*}
It follows that $f$ is a Lipschitz function. Therefore it is differentiable
almost everywhere. Use Corollary~\ref{clp2} again to get
$\leb(f[[0,1]\setminus\DD])\leq\len(\gr|[0,1]\setminus\DD)=0$. Thus
$$
  \leb(f[\DD])\geq\leb(f[0,1])-\leb(f[I\setminus\DD])
  =\leb(f[0,1])\geq\abs{f(0)-f(1)}>0,
$$
which contradicts the assumption.
\end{proof}
Let us call a set \emph{perfectly dense} if its intersection with any nonempty
open set contains a perfect set.
\begin{thm}
If $f:I\to\Rset$ is a continuous function with a \si monotone graph,
then $f$ is dif\-ferentiable at a perfectly dense set.
\end{thm}
\begin{proof}
If $f$ is constant on $I$, there is nothing to prove.
Otherwise Lemma~\ref{bigder}
yields $\leb(f[\DD(f)\cap I])>0$. Therefore $\DD(f)\cap I$ is an uncountable
Borel set. Thus it contains, by the Perfect Set Theorem, a perfect set.
\end{proof}
\begin{coro}\label{coroMM}
If $f:I\to\Rset$ is a continuous function, then $\inter\MM(f)\subs\cl{\DD(f)}$.
\end{coro}

We now present several examples illustrating that one cannot prove much
more than Theorem~\ref{thmMM} and Corollary~\ref{coroMM} about
differentiability properties of $\MM$-points. The first two examples are nowhere
differentiable functions. Note that by the above corollary and
Proposition~\ref{meager3} such a function must have a small set of
$\MM$-points:
\begin{coro}\label{coroMM2}
If $f:I\to\Rset$ is a continuous, nowhere differentiable function,
then $\MM(f)$ is meager, i.e. every monotone set $M\subs\gr$ is nowhere dense.
\end{coro}
For $y\in\Rset$ let $\norm y=\dist(y,\Zset)$.
\begin{prop}\label{function}
The function $f(y)=\sum_{k=0}^\infty2^{-k}\norm{2^{k^2}y}$ is continuous
and has no $\MM$-points.
Therefore every monotone subset of $\gr$ is meager and $f$ is nowhere
differentiable.
\end{prop}
\begin{proof}
Continuity of $f$ is easy.
Fix $y\in\Rset$, $\eps>0$ and $c>0$. We want to disprove condition~\eqref{pwm2}.
Let $n\in\Nset$ be large enough (this will be specified later).
It is easy to check that there is $i\in\{1,3\}$ such that
\begin{equation}\label{Fab2}
  \bigl|\norm{2^{n^2}y}-\norm{2^{n^2}y-\tfrac i4}\bigr|\geq\tfrac14.
\end{equation}
Set $x=y-\frac i42^{-n^2}$, $z=x+2^{-n^2}$. Clearly
if $n$ is large enough, then $x\in(y-\eps)$ and $z\in(y,y+\eps)$. We show that
$x,z$ witness failure of~\eqref{pwm2}.
\begin{enumerate}[\rm(a)]
\item If $k>n$, then $-\frac i4 2^{k^2-n^2}$ is an integer. Therefore
  $\norm{2^{k^2}x}=\norm{2^{k^2}y}=\norm{2^{k^2}z}$.
\item Since $2^{n^2}z-2^{n^2}x=1$, we have $\norm{2^{n^2}x}=\norm{2^{n^2}z}$.
\item If $k<n$ and $t=y$ or $t=z$, then
  $\bigl|\norm{2^{k^2}t}-\norm{2^{k^2}x}\bigr|\leq 2^{k^2}\abs{t-x}\leq 2^{k^2-n^2}$.
  Thus
$$
  \sum_{k<n}2^{-k}\bigl|\norm{2^{k^2}t}-\norm{2^{k^2}x}\bigr|
  \leq\sum_{k<n}2^{-k}2^{k^2-n^2}
  \leq n2^{2-3n}.
$$
\end{enumerate}
It follows that
\begin{align*}
  \abs{f(z)-f(x)}
  &\overset{(a,b)}{\leq}
  \sum_{k<n}2^{-k}\bigl|\norm{2^{k^2}z}-\norm{2^{k^2}x}\bigr|
  \overset{(c)}{\leq}
  n2^{2-3n},\\
\abs{f(y)-f(x)}
  &\overset{(a)}{\geq}
  2^{-n}\bigl|\norm{2^{n^2}y}-\norm{2^{n^2}x}\bigr|
  -\sum_{k<n}2^{-k}\bigl|\norm{2^{k^2}y}-\norm{2^{k^2}x}\bigr| \\
  &\overset{(\ref{Fab2},c)}{\geq}
  2^{-n-2}-n2^{2-3n}.
\end{align*}
Combine these estimates to get
$$
  \frac{\abs{f(y)-f(x)}}{\abs{f(z)-f(x)}+\abs{z-x}}\geq
  \frac{2^{-n-2}-n2^{2-3n}}{n2^{2-3n}+2^{-n^2}}.
$$
With a proper choice of $n$ the rightmost expression
is as large as needed, in particular greater than $c$. Therefore~\eqref{pwm2}
fails.
\end{proof}

\begin{ex}
Let $f$ be the above function. Define $g(x)=(x-\frac12)\sin\frac1{2x-1}f(x)$.
It is easy to derive from the above that $g$ has no $\MM$-points
except $x=\frac12$. Straightforward calculation of Dini derivatives
at $x=\frac12$ gives $\urd g(\frac12)=\uld g(\frac12)=\frac12$ and
$\lrd g(\frac12)=\lld g(\frac12)=-\frac12$.
Therefore $\frac12$ is an $\MM$-point (actually an $\MM_1$-point).
It also follows from Theorem~\ref{thmMM} that $g$ is differentiable at no point.
In particular $\DD(g)$ is not dense in $\MM(g)$.
\end{ex}

\begin{ex}
Let $T(x)=\sum_{n=0}^\infty 2^{-n}\norm{2^nx}$ be the \emph{Takagi function}.
The following facts can be found in~\cite{takagi:survey}.
$T$ does not possess a finite one-sided derivative at any point.
However, if $x$ is a dyadic rational, then $\rd T(x)=+\infty$ and
$\ld T(x)=-\infty$. Also $T'(x)=+\infty$ at a dense set.

It follows that the sets $\DD(T)$, $\MM_1(T)$, $\MM(T)$
as well as their complements are dense.
\end{ex}

%%%%%%%%%%%%%%%%%%%%%%%%%%%%%%%%%%%%%%%%%%%%%%%%%%%%%%%%%%%%%%%%%%%%%%%%%%%%%%%%
%%%%%%%%%%%%%%%%%%%%%%%%%%%%%%%%%%%%%%%%%%%%%%%%%%%%%%%%%%%%%%%%%%%%%%%%%%%%%%%%
\section{A non-differentiable function with a monotone graph}
\label{sec:VASEK}
%%%%%%%%%%%%%%%%%%%%%%%%%%%%%%%%%%%%%%%%%%%%%%%%%%%%%%%%%%%%%%%%%%%%%%%%%%%%%%%%
%%%%%%%%%%%%%%%%%%%%%%%%%%%%%%%%%%%%%%%%%%%%%%%%%%%%%%%%%%%%%%%%%%%%%%%%%%%%%%%%

In this section we provide an example of a continuous, almost nowhere
differentiable function on $[0,1]$ with a monotone graph.
Note that it follows from the above results that such a function necessarily
have the following properties:
\begin{itemize}
\item Every point of $[0,1]$ is an $\MM$-point,
\item the function is almost nowhere approximately differentiable,
\item almost all points are knot points (actually approximate knot points),
\item the function has a derivative at a perfectly dense set.
\end{itemize}
\begin{thm}\label{vas1}
For any $c>1$ there is a continuous, almost nowhere differentiable
function $f:[0,1]\to\Rset$ with a symmetrically $c$-monotone graph.
\end{thm}
The proof is a bit involved.
The function $f$ we construct satisfies condition $\mathsf P_1$. That is enough,
because given any $c>1$, the function $x\mapsto(c-1)f(x)$ satisfies obviously
condition $\mathsf P_{c-1}$ and is thus by Lemma~\ref{L4} $c$-monotone.
We first construct the function and then prove its properties in
a sequence of lemmas.

\subsection*{Construction of the function}
The function $f$ is defined as a limit of a sequence of
piecewise linear continuous functions $f_n:[0,1]\to [0,1]$ that we now define.

We recursively specify finite sets
$\mc A_n=\{a_n^k:k=0,\dots,r_n\}\subs[0,1]$ such that
$$
  0=a_n^0<a_n^1<\dots<a_n^{r_n}=1
$$
and values of $f_n$ at each point of $\mc A_n$.
The function $f_n$ is then defined as the unique function that is linear between
consecutive points of $\mc A_n$.

For $n=0$ put $\mc A_0=\{0,1\}$ and $f_0(0)=f_0(1)=0$.

The induction step:
Suppose $f_n$ and $\mc A_n=\{a_n^k:k=0,\dots,r_n\}$ are constructed.
Let $k<r_n$ be arbitrary.
For $l=0,\dots,5$ set
$x_l=\frac l5 a_n^{k+1}+(1-\frac l5)a_n^k$.
%$x_l=\dfrac{la_n^{k+1}+(5-l)a_n^k}{5}$.

\noindent
If $f_n(a_n^k)=f_n(a_n^{k+1})$, set $A_n^k=\{x_l:l=1,\dots,5\}$ and
\begin{align*}
  f_{n+1}(x_0)&=f_{n+1}(x_1)=f_{n+1}(x_4)=f_{n+1}(x_5)=f_n(a_n^k),\\
  f_{n+1}(x_2)&=f_{n+1}(x_3)=f_n(a_n^k)+\tfrac16\abs{a_n^{k+1}-a_n^k}.
\intertext{If $f_n(a_n^k)\neq f_n(a_n^{k+1})$, set $A_n^k=\{x_0,x_1,x_4,x_5\}$ and}
  f_{n+1}(x_0)&=f_n(a_n^k),\\
  f_{n+1}(x_5)&=f_n(a_n^{k+1}),\\
  f_{n+1}(x_1)&=f_{n+1}(x_4)=\tfrac12\bigl(f_n(a_n^k)+f_n(a_n^{k+1})\bigr)
\end{align*}
and let $\mc A_{n+1}=\bigcup_{k=0}^{r_n-1}A_n^k$.

\begin{lem}\label{L2}
Let $n\in\Nset$ and $k<r_n$. Then the following holds:
\begin{enum}
\item If $k>0$, then $\abs{a_n^{k-1}-a_n^k}
       \leq 3\abs{a_n^{k+1}-a_n^k}\leq 9\abs{a_n^{k-1}-a_n^k}$,
\item $\abs{a_{2n}^{k+1}-a_{2n}^k}\leq \left(\frac{3}{25}\right)^n$,
\item $\abs{a_{2n+1}^{k+1}-a_{2n+1}^k}\leq \frac15\left(\frac{3}{25}\right)^n$,
\item $\abs{a_n^{k+1}-a_n^k}\geq \left(\frac15\right)^n$,
\item $f_i(a_n^k)=f_n(a^k_n)$ if $i\geq n$,
\item $\abs{f_n(a_n^{k+1})-f_n(a_n^k)}\leq\frac16\left(\frac{1}{2}\right)^n$,
\item $\frac{\abs{f_n(a_n^k)-f_n(a_n^{k+1})}}{\abs{a_n^k-a_n^{k+1}}}=
  0$ or $\frac{\abs{f_n(a_n^k)-f_n(a_n^{k+1})}}{\abs{a_n^k-a_n^{k+1}}}\geq
  \frac56$,
\item if $i>0$ and $x\in [a_n^k,a_n^{k+1}]$, then
  \begin{align*}
    \min\bigl(f_n(a^k_n),f_n(a_n^{k+1})\bigr)&\leq f_{n+i}(x)\\
    &\leq
    \max\bigl(f_n(a^k_n),f_n(a_n^{k+1})\bigr)+\abs{a^{k+1}_n-a_n^k}
    \,\sum_{j=1}^i6^{-j},
  \end{align*}
\item if $i>0$, $x\in (a_n^k,a_n^{k+1})$ and $f_n(a_n^k)\neq f_n(a_n^{k+1})$, then
  $$
    f_{n+i}(x)<\max\bigl(f_n(a^k_n),f_n(a_n^{k+1})\bigr),
  $$
\item $f_n$ is continuous and $f_n(x)\in [0,1]$ for all $x\in [0,1]$.
\end{enum}
\end{lem}
\begin{proof}
(i)--(v) follows right away from the construction of functions $f_n$.
(vi) can be easily proved from the construction using (ii) and (iii).

(vii): Case $n=0$ is trivial.
Assume (vii) holds for some $n\geq 0$ and we prove it for $n+1$.
Let $i<r_{n+1}$ be arbitrary. There exists $k<r_n$ such that
$a_{n+1}^i,a_{n+1}^{i+1}\in [a^k_n,a^{k+1}_n]$.

\textbullet{}
If $\dfrac{\abs{f_n(a_n^k)-f_n(a_n^{k+1})}}{\abs{a_n^k-a_n^{k+1}}}=0$, then
$$
  \frac{\abs{f_{n+1}(a_{n+1}^i)-f_{n+1}(a_{n+1}^{i+1})}}{\abs{a_{n+1}^i-a_{n+1}^{i+1}}}=0
  \quad\vee\quad
  \frac{\abs{f_{n+1}(a_{n+1}^i)-f_{n+1}(a_{n+1}^{i+1})}}{\abs{a_{n+1}^i-a_{n+1}^{i+1}}}
  =\frac56.
$$

\textbullet{}
If $\dfrac{\abs{f_n(a_n^k)-f_n(a_n^{k+1})}}{\abs{a_n^k-a_n^{k+1}}}\neq 0$,
then
$$
  \frac{\abs{f_{n+1}(a_{n+1}^i)-f_{n+1}(a_{n+1}^{i+1})}}{\abs{a_{n+1}^i-a_{n+1}^{i+1}}}=0
$$
or
$$
  \frac{\abs{f_{n+1}(a_{n+1}^i)-f_{n+1}(a_{n+1}^{i+1})}}{\abs{a_{n+1}^i-a_{n+1}^{i+1}}}
  =\frac{5}{2}\frac{\abs{f_n(a_n^k)-f_n(a_n^{k+1})}}{\abs{a_n^k-a_n^{k+1}}}\geq \frac56.
$$
(viii): The first inequality is obvious. The second inequality is proved
by induction over $i$.
Case $i=1$ easily follows from the construction. Suppose that this statement is true
for $i=p$. We show that it is also true for $i=p+1$.
Find $l<r_{n+1}$ such that $x\in [a_{n+1}^l,a_{n+1}^{l+1}]$ and
use the induction hypothesis to compare $f_n(a^k_n),f_n(a^{k+1}_n)$ with
$f_{n+1}(a^l_{n+1}),f_{n+1}(a^{l+1}_{n+1})$ (which is the case $i=1$) and
$f_{n+1}(a^l_{n+1}),f_{n+1}(a^{l+1}_{n+1})$ with $f_{n+p+1}(x)$ (which is the case $i=p$).

(ix): This is similar to (viii). Case $i=1$ easily follows from
the construction. Proceed by induction: Assume that the statement is true for
$i=p$. We show that it is also true for $i=p+1$. Find $l<r_{n+1}$ such that
$x\in [a_{n+1}^l,a_{n+1}^{l+1}]$.

If $f(a_{n+1}^l)\neq f(a_{n+1}^{l+1})$ then use the statement to compare
$f_n(a^k_n),f_n(a^{k+1}_n)$ with
$f_{n+1}(a^l_{n+1}),f_{n+1}(a^{l+1}_{n+1})$ (which is the case $i=1$) and
$f_{n+1}(a^l_{n+1}),f_{n+1}(a^{l+1}_{n+1})$ with $f_{n+p+1}(x)$ (which is the case $i=p$).

If $f(a_{n+1}^l)=f(a_{n+1}^{l+1})$ then by the construction and (vii) we have
\begin{align*}
  \frac{25}{36}\abs{a_{n+1}^l-a_{n+1}^{l+1}}&=\frac{5}{12}\abs{a_n^{k+1}-a_n^k}\\
  &\leq\frac12\max\bigl(f_n(a^k_n),f_n(a_n^{k+1})\bigr)-\min\bigl(f_n(a^k_n),f_n(a_n^{k+1})
  \bigr)\\
  &=\max\bigl(f_n(a^k_n),f_n(a_n^{k+1})\bigr)-f_{n+1}(a_{n+1}^l).
\end{align*}
By (viii) we have
$$
  f_{n+p+1}(x))\leq f_{n+1}(a_{n+1}^l)+\frac15\abs{a_{n+1}^l-a_{n+1}^{l+1}}.
$$
Thus $f_{n+p+1}(x)<\max\bigl(f_n(a^k_n),f_n(a_n^{k+1}\bigr).$

(x) can be easily proved from the construction using (viii).
\end{proof}
\begin{lem}\label{L3}
The functions $f_i$ satisfy condition $\mathsf P_1$ for every $i$.
\end{lem}
\begin{proof}
Let $x<y\in[0,1]$ and $i$ be arbitrary such that $f_i(x)=f_i(y)$. We show that
\begin{equation}\label{L1}
\max_{x\leq t\leq y}\abs{f_i(x)-f_i(t)}\leq |x-y|.
\end{equation}
Since $f_i$ is piecewise linear, level sets are finite.
We may thus assume that there is no $w\in(x,y)$ such that $f_i(w)=f_i(x)$.
%Otherwise, we prove \eqref{L1} for couples $(x,w)$ and $(w,y)$ instead of $(x,y)$.
Let $z\in(x,y)$ be such that
$\max_{x\leq t\leq y}\abs{f_i(x)-f_i(t)}=\abs{f_i(x)-f_i(z)}$.
The case $f_i(x)=f_i(z)$ is trivial. We may thus assume that either $f_i(x)<f_i(z)$
or $f_i(x)>f_i(z)$.

Assume first $f_i(x)<f_i(z)$. By the construction of $f_i$ we can find minimal
$n\leq i$ and $k<r_n-1$ such that $z\in(a_n^k,a_n^{k+1})\subset(x,y)$ and
$f_i(a_n^k)=f_i(a_n^{k+1})\in(f_i(x),f_i(z)]$. By Lemma \ref{L2}(v) we have
$f_n(a_n^k)=f_n(a_n^{k+1})=f_i(a_n^k)$. We show that $f_n(a_n^{k-1})<f_n(a_n^k)$.

Suppose the contrary: $f_n(a_n^{k-1})>f_n(a_n^k)$. By Lemma \ref{L2}(viii) we have
$f_i(t)\geq f_n(a_n^k)$ for all $t\in(a_n^{k-1},a_n^k)$.
So, $x\notin[a_n^{k-1},a_n^k]$. Thus $a_n^{k-1}\in(x,y)$.
By Lemma \ref{L2}(vii) and (i) we have
$$
  f_n(a_n^{k-1})\geq f_n(a_n^k)+\frac56\abs{a_n^{k-1}-a_n^k}
  \geq f_n(a_n^k)+\frac{5}{18}\abs{a_n^k-a_n^{k+1}}.
$$
By Lemma \ref{L2}(viii) we have $f(z)\leq f_n(a_n^k)+\frac15\abs{a_n^k-a_n^{k+1}}$.
 Thus $f_i(z)<f_i(a_n^{k-1})$, which contradicts that $f_i(t)\leq f_i(z)$ for all $t\in(x,y)$.

Similarly, we have $f_n(a_n^{k+1})>f_n(a_n^{k+2})$. %(figure4).

By the construction we have that there exists $l<r_{n-1}$ such that
$a_{n-1}^l=a_n^{k-2}$, $a_{n-1}^{l+1}=a_n^{k+3})$ and $f_n(a_{n-1}^l)=f_n(a_{n-1}^{l+1})$.
By the minimality of $n$ we have $(x,y)\not\supset(a_{n-1}^l,a_{n-1}^{l+1})$.
Thus $x\in[a_{n-1}^l,a_{n-1}^{l+1}]$ or $y\in[a_{n-1}^l,a_{n-1}^{l+1}]$.
We can assume $x\in[a_{n-1}^l,a_{n-1}^{l+1}]$.
By Lemma \ref{L2}(viii) and $x,z\in[a_{n-1}^l,a_{n-1}^{l+1}]$
we have
$$
  \max_{x\leq t\leq y}\abs{f_i(x)-f_i(t)}=
  \abs{f_i(x)-f_i(z)}\leq\frac15\abs{a_{n-1}^l-a_{n-1}^{l+1}}=
  \abs{a_n^k-a_n^{k+1}}\leq\abs{x-y}.
$$
Now assume $f_i(x)>f_i(z)$. By the construction of $f_i$ and Lemma \ref{L2}(viii)
we can find minimal $n\leq i$ and $k<r_n-1$ such that $a_n^k,a_n^{k+1}\in(x,y)$ and
$f_i(a_n^k)=f_i(a_n^{k+1})=f_i(z)$. By Lemma \ref{L2}(v) we have
$f_n(a_n^k)=f_n(a_n^{k+1})=f_i(z)$. Since $f_i(t)\geq f_n(a_n^k)$ for all $t\in(x,y)$
and Lemma \ref{L2}(ix) we have $f_n(a_n^{k-1}),f_n(a_n^{k+2})>f_n(a_n^k)$.
%(see figure 6).
By the construction there is no $l<r_{n-1}$ such that
$(a_{n-1}^l,a_{n-1}^{l+1})\supset(a_n^{k-1},a_n^{k+2})$.
Thus there are two possible cases:
\begin{enum}
\item There exists $l<r_{n-1}$ such that $a_{n-1}^l=a_n^{k+1}$ and
$f(a_{n-1}^{l-1})=f(a_{n-1}^l)$.%(figure 7)
\item There exists $l<r_{n-1}$ such that $a_{n-1}^l=a_n^k$ and
$f(a_{n-1}^{l+1})=f(a_{n-1}^l)$. %(figure 8)
\end{enum}
We prove only (i), as the case (ii) is similar. By minimality of $n$ we have
$x\in[a_{n-1}^{l-1},a_{n-1}^l]$. Lemma \ref{L2}(viii) yields
\begin{align*}
\max_{x\leq t\leq y}\abs{f_i(x)-f_i(t)}&=
  \abs{f_i(x)-f_n(a_{n-1}^l)} \\
  &\leq\frac15\abs{a_{n-1}^{l-1}-a_{n-1}^l}=
  \abs{a_n^k-a_n^{k+1}}<\abs{x-y}. \qedhere
\end{align*}
\end{proof}
\begin{lem}
The sequence $\{f_n\}$ is uniformly Cauchy.
\end{lem}
\begin{proof}
Fix $n\in\Nset$ and let $k<r_n$. If $a_n^k\leq x\leq a_n^{k+1}$, then
by construction of $f_{n+1}$
$$
  \abs{f_{n+1}(x)-f_n(x)}
  \leq \frac16|a_n^{k+1}-a_n^k|+\frac{3}{10}|f_n(a_n^{k+1})-f_n(a_n^k)|.
$$
Estimate $|a_n^{k+1}-a_n^k|$ using Lemma~\ref{L2}(ii) and (iii) and
$|f_n(a_n^{k+1})-f_n(a_n^k)|$ using Lemma~\ref{L2}(vi)
and combine the estimates to get
$$
  \frac16\abs{a_n^{k+1}-a_n^k}+\frac{3}{10}\abs{f_n(a_n^{k+1})-f_n(a_n^k)}\leq
  2^{-n}.
$$
Thus $\abs{f_{n+1}(x)-f_n(x)}\leq 2^{-n}$, irrespective of the particular $k$.
Since the intervals $[a_n^k,a_n^{k+1}]$, $k<r_n$, cover $[0,1]$, we have
$\abs{f_{n+1}(x)-f_n(x)}\leq 2^{-n}$ for all $x$, which is clearly enough.
\end{proof}
This lemma lets us define $f=\lim_{n\to\infty}f_n$. We claim that thus defined
$f$ is the required function. It is of course continuous.
By Lemma \ref{L3} the functions $f_n$ satisfy condition
$\mathsf P_1$.
It is easy to check that since $f$ is a limit of $f_n$'s, it satisfies
$\mathsf P_1$ as well. We thus have
\begin{prop}\label{vas11}
$f$ is a continuous function satisfying $\mathsf P_1$.
\end{prop}
It remains to show that $f$ fails to have a derivative at almost all points.
For $n\in\Nset$ define
\begin{align*}
  A_n&=\cl{\{x\in [0,1]:f_n'(x)=0\}},\\
  B_n&= \cl{[0,1]\setminus A_n},\\
  B&=\bigcup_{i\in\Nset}\bigcap_{n\geq i}B_n,\\
  D&=\left\{x\in [0,1];\ \forall n\in\Nset: x\cdot5^n\pmod 1\notin
      \left(\tfrac{1}{5},\tfrac{4}{5}\right)\right\}.
\end{align*}
\begin{lem}\label{vas22}
$\leb(B)=0$.
\end{lem}
\begin{proof}
For every $n$ set
$\mc M_n=\bigl\{i<r_n:f_n'
\bigl(\frac{a^i_n+a^{i+1}_n}{2}\bigr)\neq0\bigr\}$.
It is easy to see that
$$
  B=\bigcup_{n\in\Nset}\bigcup_{i\in\mc M_n}
    \{x\cdot\abs{a^{i+1}_n-a^i_n}+a^i_n:x\in D\}
$$
and since obviously $\leb(D)=0$, we are done.
\end{proof}
\begin{prop}\label{vas33}
\begin{enum}
%\item $\DD(f)\subs B$
\item If $x\notin B$, then $\rd f(x)$ and $\ld f(x)$ do not exist.
\item If $x\in B$, then at least one of the Dini derivatives of $f$ at $x$
is infinite.
%$\infty\in\{\urd f(x),\uld f(x),|\lrd f(x)|,|\lld f(x)|\}$.
\end{enum}
\end{prop}
\begin{proof}
(i): Let $x\notin B$. We show that $\rd f(x)$
does not exists, the proof for $\ld f(x)$ is similar.
Let $\delta>0$ be arbitrary. Since $x\notin B$ there exist $n\in\Nset$
and $k_i<r_{n+i}$ such that $x\in (a^{k_0}_n,a^{k_0+1}_n)$,
$f_n(a^{k_0}_n)=f_n(a^{k_0+1}_n)$,
$\abs{a^{k_0+1}_n-a^{k_0}_n}<\delta$ and $a^{k_i}_{n+i}=a^{k_0}_n$ for all $i\in\Nset$.
By the construction of functions $f_n$ we have
$f_{n+i}(a^{k_i}_{n+i})=f_{n+i}(a^{k_i+1}_{n+i})$.
Since $x\neq a^{k_0}_n$ there exists $i\in\Nset$ such that
$x\notin [a^{k_i}_{n+i},a^{k_i+1}_{n+i}]$.
We may assume that $x\notin (a^{k_1}_{n+1},a^{k_1+1}_{n+1})$.
By Lemma \ref{L2}(v) and (viii) we have
\begin{align*}
  \left|\tfrac{f(a_{n+2}^{k_2+3})-f(x)}{a_{n+2}^{k_2+3}-x}
  -\tfrac{f(a_{n+2}^{k_2+4})-f(x)}{a_{n+2}^{k_2+4}-x}\right|
    &=\left|\tfrac{f_{n+2}(a_{n+2}^{k_2+3})-f(x)}{a_{n+2}^{k_2+3}-x}
    -\tfrac{f_{n+2}(a_{n+2}^{k_2+4})-f(x)}{a_{n+2}^{k_2+3}-x}\right|\\
    &\geq\left|\tfrac{f_{n+2}(a_{n+2}^{k_2+3})-f(x)}{a_{n+2}^{k_2+3}-x}
    -\tfrac{f_{n+2}(a_{n+2}^{k_2+4})-f(x)}{a_{n+2}^{k_2+3}-x}\right|\\
    &\geq\left|\tfrac{f_{n+2}(a_{n+2}^{k_2+3})-f_{n+2}(a_{n+2}^{k_2+4})}
    {|a^{k_0+1}_n-a^{k_0}_n|}\right|\geq\frac{1}{30}.
\end{align*}
Thus, $\rd f(x)$ does not exists.

(ii): Since $x\in B$ there exist $n\in\Nset$ and $k_i<r_{n+i},\ i\in\Nset$,
such that
\begin{itemyze}
\item $x\in [a^{k_i}_{n+i},a^{k_i+1}_{n+i}]$ for all $i\in\Nset$,
\item $f_n(a^{k_0}_n)=f_n(a^{k_0+1}_n)$,
\item $f_{n+i}(a^{k_i}_{n+i})\neq f_{n+i}(a^{k_i+1}_{n+i})$ for all $i>0$.
\end{itemyze}
By the construction of functions $f_n$ we have, for all $i>0$,
$$
  \left|\frac{f_{n+i}(a^{k_i+1}_{n+i})-f_{n+i}(a^{k_i}_{n+i})}{a^{k_i+1}_{n+i}-a^{k_i}_{n+i}}
  \right|=\frac56\left(\frac{5}{2}\right)^{i-1}.
$$
By Lemma \ref{L2}(v) we have, for all $i>0$,
$$
  \left|\frac{f(a^{k_i+1}_{n+i})-f(x)}{a^{k_i+1}_{n+i}-x}
  \right|\geq\frac56\left(\frac{5}{2}\right)^{i-1}
$$
or
$$
  \left|\frac{f(x)-f(a^{k_i}_{n+i})}{x-a^{k_i}_{n+i}}
  \right|\geq\frac56\left(\frac{5}{2}\right)^{i-1},
  $$
which is clearly enough.
\qedhere
\end{proof}
Theorem~\ref{vas1} now follows from Proposition~\ref{vas11}, Lemma~\ref{vas22}
and Proposition~\ref{vas33}.

%%%%%%%%%%%%%%%%%%%%%%%%%%%%%%%%%%%%%%%%%%%%%%%%%%%%%%%%%%%%%%%%%%%%%%%%%%%%%%%%
%%%%%%%%%%%%%%%%%%%%%%%%%%%%%%%%%%%%%%%%%%%%%%%%%%%%%%%%%%%%%%%%%%%%%%%%%%%%%%%%
\section{$\MM_1$-points}
\label{sec:m1}
%%%%%%%%%%%%%%%%%%%%%%%%%%%%%%%%%%%%%%%%%%%%%%%%%%%%%%%%%%%%%%%%%%%%%%%%%%%%%%%%
%%%%%%%%%%%%%%%%%%%%%%%%%%%%%%%%%%%%%%%%%%%%%%%%%%%%%%%%%%%%%%%%%%%%%%%%%%%%%%%%

It turns out that being an $\MM_1$-point is a particularly simple and strong property:
it is, modulo negligible set, equivalent to differentiability.
We begin with an elementary lemma.
\begin{lem}\label{lemMM1a}
Let $f:I\to\Rset$ be continuous and $y\in I$. Suppose $\eps>0$ is such that
condition~\eqref{pwm} holds with $c=1$.
If there is $x\in(y-\eps,y)$ such that $f(x)>f(y)$, then
$\urd f(y)\leq\frac{y-x}{f(x)-f(y)}$.
\end{lem}
\begin{proof}
Let $C$ be the open disc centered at $\psi(x)$ whose boundary circle passes
through $\psi(x)$. If $z\in(y,y+\eps)$, then $\psi(z)\notin C$.
Therefore $\urd f(y)$ is less than or equal to the slope of the line tangent
to $C$ at $\psi(y)$. This slope is clearly equal to $\frac{y-x}{f(x)-f(y)}$,
as required.
\end{proof}
\begin{coro}
Let $f:I\to\Rset$ be continuous and $y\in I$ an $\MM_1$-point.
If $\lld f(y)<0$, then $\urd f(x)\leq \frac1{\abs{\lld f(y)}}$.
\end{coro}
\begin{thm}\label{thmMM1}
If $f:I\to\Rset$ is continuous, then there is a set $E\subs I$ such that
$\len(\gr|E)=0$ and $\DD(f)\subs\MM_1(f)\subs\DD(f)\cup E$.
In particular, $f$ is differentiable at almost every $\MM_1$-point.
\end{thm}
\begin{proof}
If $f$ has a derivative, finite or infinite, at $y$, then there is obviously
$\eps>0$ such that if $y-\eps\leq x<y<z\leq y+\eps$, then
the angle spanned by the vectors $\psi(x)-\psi(y)$ and $\psi(z)-\psi(y)$
is obtuse and consequently
\begin{equation}\label{Hr999x}
  \abs{\psi(y)-\psi(x)}\leq\abs{\psi(z)-\psi(x)},
\end{equation}
which is nothing but condition~\eqref{pwm} with $c=1$.
Hence $y$ is an $\MM_1$-point.

To prove the latter inclusion we employ the famous Denjoy--Young--Saks Theorem,
cf.~\cite[IX(4.2)]{MR0167578}: \emph{There is a set $E\subs I$ such that
$\len(\gr|E)=0$ and for every point $x\notin E$ one of the following
cases occurs:
\emph{(a)} $f'(x)$ exists,
\emph{(b)} $x$ is a knot point,
\emph{(c)} $\urd f(x)=-\lld f(x)=\infty$ and $\uld f(x)=\lrd f(x)$ are finite,
\emph{(d)} $\uld f(x)=-\lrd f(x)=\infty$ and $\urd f(x)=\lld f(x)$ are finite.}

Suppose for contrary that there is $y\in\MM_1(f)\setminus(\DD\cup E)$.
Then one of cases (b), (c), (d) occurs.
Since (d) obtains from (c) by reversing the $y$-axis,
we only have to consider (b) and (c). In either case, $\lld f(y)=-\infty$ and
$\urd f(y)=\infty$. The above corollary yields $\urd f(x)\leq0$: a contradiction.
\end{proof}

The set $\gr|\DD(f)$ is, by this theorem and Proposition~\ref{meager2}(iii),
a countable union of closed $1$-monotone sets. An easy symmetry argument gives
a bit more:
\begin{coro}\label{Hr999}
If $f:I\to\Rset$ is continuous, then $\gr|\DD(f)$ admits a
countable cover by symmetrically $1$-monotone sets.
\end{coro}

$1$-monotone graphs behave particularly nice:
\begin{thm}\label{1-mono}
If $I$ is compact and $f:I\to\Rset$ is continuous with a $1$-monotone graph
, then $f$ is of bounded variation.
\end{thm}
\begin{proof}
Set $A=\{x\in I: \forall y\in[0,x)\ f(y)-f(x)<x-y\}$ and let $g(x)=f(x)+x$.
Obviously $A=\{x\in I:\forall y\in[0,x)\ g(y)<g(x)\}$,
hence $g$ is increasing on $A$.
By~\cite[VII(4.1)]{MR0167578} there is a non-decreasing extension
$g^*:I\to\Rset$ of $g$.
If follows that $f^*(x)=g^*(x)-x$ is of bounded variation on $[0,1]$
and clearly $f^*(x)=f(x)$ for all $x\in A$.
Therefore $\len(\gr|A)\leq\len(\gr^*)<\infty$.
The same argument shows that letting
$B=\{x\in I: \forall y\in[0,x)\ f(x)-f(y)<x-y\}$ we have
$\len(\gr|B)\leq\len(\gr^*)<\infty$.

Now suppose $x\notin A$, i.e.~there is $y<x$ such that $f(y)-f(x)\geq x-y$.
Lemma~\ref{lemMM1a} yields $\urd f(x)\leq1$. The same argument shows that if
$x\notin B$, then $\lrd f(x)\geq-1$.
In summary, if $x\notin A\cup B$, then $\abs{\urd f(x)}\leq1$.
By the remark following~\cite[IX(4.6)]{MR0167578}
$$
  \len(\gr|I\setminus(A\cup B))\leq
  \int_{I\setminus(A\cup B)}\sqrt{1+(\urd f(x))^2}\leq
  \sqrt2\leb(I)<\infty.
$$
Altogether
$\len(\gr)\leq\len(\gr|A)+\len(\gr|B)+\len(\gr|I\setminus(A\cup B))<\infty$.
In particular, $f$ is of bounded variation.
\end{proof}
Since every nondecreasing function has a $1$-monotone graph, we have the
following characterization of bounded variation.
\begin{coro}\label{BVchar}
A continuous function $f:[0,1]\to\Rset$ is of bounded variation if and only if
it is a sum of two continuous functions with $1$-monotone graphs.
\end{coro}

%%
%For each $n$ put
%$$
%  D_n=\{y\in\DD(f):\eps_y>2^{-n}\}
%$$
%and cover each $D_n$ with finitely many intervals $I_{nk}$ of length $2^{-n}$.
%If $x<y<z$ are in $D_n\cap I_{nk}$, then clearly $y-\eps_y\leq x<y<z\leq y+\eps_y$
%and thus \eqref{Hr999x} holds. In summary, $\gr|(D_n\cap I_{nk})$ is symmetrically
%$1$-monotone for each $n$ and $k$ and the family $\{\gr|(D_n\cap I_{nk})\}$
%covers $\gr|\DD(f)$.
%%\end{proof}
%%

%%%%%%%%%%%%%%%%%%%%%%%%%%%%%%%%%%%%%%%%%%%%%%%%%%%%%%%%%%%%%%%%%%%%%%%%%%%%%%%%
%%%%%%%%%%%%%%%%%%%%%%%%%%%%%%%%%%%%%%%%%%%%%%%%%%%%%%%%%%%%%%%%%%%%%%%%%%%%%%%%
\section{An absolutely continuous function with a non-\si monotone graph}
\label{sec:MICHAL}
%%%%%%%%%%%%%%%%%%%%%%%%%%%%%%%%%%%%%%%%%%%%%%%%%%%%%%%%%%%%%%%%%%%%%%%%%%%%%%%%
%%%%%%%%%%%%%%%%%%%%%%%%%%%%%%%%%%%%%%%%%%%%%%%%%%%%%%%%%%%%%%%%%%%%%%%%%%%%%%%%

If $f:I\to\Rset$ is absolutely continuous, then it is differentiable almost
everywhere and, moreover, the set $\gr|I\setminus\DD(f)$ is of length zero.
Thus Corollary~\ref{Hr999} yields:
\begin{coro}
If $f:I\to\Rset$ is absolutely continuous, then there is a countable family
$\{M_n\}$ of symmetrically $1$-monotone sets such that\label{Hr1000}
$$
  \len\Bigl(\gr\setminus\bigcup\nolimits_{n\in\Nset}M_n\Bigr)=0.
$$
\end{coro}
We want to show that this fact cannot be sharpened by providing
an example of an absolutely continuous function whose graph is
not \si monotone.

Recall the notion of strong porosity, as defined in~\cite{HrZi}.
A set $X\subs\Rset^2$ is termed \emph{strongly porous} if there is $p>0$
such that for any $x\in\Rset^2$ and any $r\in(0,\diam X)$
there is $y\in\Rset^2$ such that $B(y,pr)\subs B(x,r)\setminus X$.
The constant $p$ is termed a \emph{porosity constant} of $X$.
As proved in~\cite[Theorem 4.2]{HrZi}, every monotone set in $\Rset^2$ is
strongly porous. More information on porosity properties of monotone sets in $\Rset^n$
can be found in~\cite{ZinFrac}.

M.~Zelen\' y~\cite{2177418} found an example of an absolutely continuous function
whose graph is not \si porous%
\footnote{See~\cite{2177418} or~\cite{943561,2201041} for the definition
of \si porous.},
and since a countable union of strongly porous sets is \si porous,
we have, in view of~\cite[Theorem 4.2]{HrZi} mentioned above,
the following theorem.
\begin{thm}
There is an absolutely continuous function on $[0,1]$ whose graph is not \si monotone.
\end{thm}
Zelen\' y's example is rather involved. We provide another example
that is much simpler and moreover it exhibits that the implication
monotone $\Rightarrow$ strongly porous cannot be reversed even for graphs.

\begin{thm}
There is an absolutely continuous function $f:[0,1]\to\Rset$
whose graph is strongly porous, but every monotone subset of $\gr$ is nowhere
dense. In particular, $\gr$ is not \si monotone.\label{Hrrrr}
\end{thm}
The function is built of single peak functions. Let
$\norm x=\dist(x,\Rset\setminus[-1,1])$.
Fix two sequences of positive reals $\seq{a_n}$ and $\seq{b_n}$.
Suppose that $\sum_na_n<\infty$ and let the sequence $\seq{q_n}$ enumerate
all rationals within $[0,1]$. The following formula defines
a real-valued function $f:[0,1]\to\Rset$.
$$
  f(x)=\sum_{n\in\Nset}a_n \Bigl\|\frac{x-q_n}{b_n}\Bigr\|
$$
We will show that with a proper choice of the two sequences the function $f$
possesses the required properties.

For simplicity stake write $f_n(x)=a_n \bigl\|\frac{x-q_n}{b_n}\bigr\|$
and $s_n=\frac{a_n}{b_n}$.
\begin{lem}
If $\sum_na_n<\infty$, then $f$ is absolutely continuous.   \label{HrA}
\end{lem}
\begin{proof}
Fix $\eps>0$. Choose $m\in\Nset$ such that
$\sum_{n>m}a_n\leq\eps$ and let
$\delta= \frac{\eps}{\sum\limits_{n\leq m}s_n}$.
Suppose $x_0<y_0<x_1<y_1<\dots<x_k<y_k$ satisfy
$\sum\limits_{i=0}^k y_i-x_i<\delta$.
Since $\abs{f_n(x_i)-f_n(y_i)}\leq s_n(y_i-x_i)$ for all $i$ and $n$,
we have, for all $n$,
\begin{align}
  \sum_{i=0}^k\abs{f_n(x_i)-f_n(y_i)}&\leq\sum_{i=0}^k s_n(y_i-x_i)
  <\delta s_n \label{Hr4}
\intertext{%
and since the function $f_n$ is unimodal and ranges between $0$ and $a_n$, also}
  \sum_{i=0}^k\abs{f_n(x_i)-f_n(y_i)}&\leq 2a_n. \label{Hr5}
\end{align}
Use~\eqref{Hr4} for $n\leq m$ and~\eqref{Hr5} for $n>m$ to get
\begin{align*}
\sum_{i=0}^k\abs{f(x_i)-f(y_i)}
  &\leq\sum_{n\leq m}\sum_{i=0}^k\abs{f_n(x_i)-f_n(y_i)}+
  \sum_{n>m}\sum_{i=0}^k\abs{f_n(x_i)-f_n(y_i)}\\
  &\overset{(\ref{Hr4},\ref{Hr5})}{<}
  \sum_{n\leq m}\delta s_n+\sum_{n>m} 2a_n
    \leq
  \eps+2\eps=3\eps,
\end{align*}
the last inequality by the choice of $\eps$ and $\delta$.
\end{proof}
\begin{lem}\label{HrB}
If $\lim\limits_{m\to\infty}\frac{\sum_{n>m}a_n}{a_m}=0$ and
$\lim\limits_{m\to\infty}\frac{\sum_{n<m} s_n}{s_m}=0$, then $\MM(f)$ is meager.
\end{lem}
\begin{proof}
It is clear that if $s_m>2c$, then the points $q_m-b_m<q_m<q_m+b_m$ witness that
the graph $f_m$ is not $c$-monotone.
We want to show that the same argument works for the entire sum $f=\sum_nf_n$.
The former condition ensures that the terms $f_n$, $n>m$, contribute to the sum
negligible quantities because of their small magnitudes.
The latter condition ensures that also the terms $f_n$, $n<m$, are negligible
because of their small slopes.

Write
$$
  \eps_m=\frac{\sum_{n>m}a_n}{a_m}+\frac{\sum_{n<m}s_n}{s_m}
  =\frac{\sum_{n>m}^\infty a_n+b_m\sum_{n<m} s_n}{a_m}.
$$
According to Propositions~\ref{intM},~\ref{meager2} and Lemma~\ref{baire}
it is enough to show that $\gr|I$ is monotone for no interval $I$.

Fix $c>0$ and an interval $I$.  The hypotheses ensure that $\eps_m\to 0$ and
$s_m\to\infty$. Therefore if $m$ is large enough $m$, then
\begin{equation}\label{Hr7}
  \frac{1-\eps_m}{2(\frac1{s_m}+\eps_m)}>c.
\end{equation}
Choose such na $m$ subject to $[q_m{-}b_m,q_m{+}b_m]\subs I$.
Write $x=q_m{-}b_m$, $z=q_m{+}b_m$.
If we succeed to prove that
\begin{equation}\label{Hr8}
  \abs{\psi(z)-\psi(q_m)}>c\abs{\psi(z)-\psi(x)},
\end{equation}
we will be done, because the points $x<q_m<z$ will witness that
$\gr|I$ is not $c$-mono\-tone.
Estimate the term on the left
\begin{align*}
  \abs{\psi(z)-\psi(q_m)}
  &\geq\abs{f(z)-f(q_m)}
  \geq\abs{f_m(z)-f_m(q_m)}-
  \sum_{n\neq m}\abs{f_n(z)-f_n(q_m)}\\
  &\geq a_m-
  \Bigl(\sum_{n<m}\abs{f_n(z){-}f_n(q_m)}
  +\sum_{n>m}\abs{f_n(z){-}f_n(q_m)}\Bigr)\\
  &\geq a_m-\Bigl(\sum_{n<m}s_nb_m+\sum_{n>m}a_n\Bigr)
  =a_m-\eps_ma_m=a_m(1-\eps_m),\\
\intertext{and the term on the right}
  \abs{\psi(z)-\psi(x)}
  &\leq
  2b_m+\abs{f(z)-f(y)}\\
  &\leq 2b_m+\sum_{n<m}\abs{f_n(z){-}f_n(x)}
  +\sum_{n>m}\abs{f_n(z){-}f_n(x)}\\
  &\leq 2b_m+2b_m\sum_{n<m}s_n+\sum_{n>m}a_n
  \leq 2b_m+2\Bigl(b_m\sum_{n<m}s_n+\sum_{n>m}a_n\Bigr)\\
  &\leq 2(b_m+\eps_ma_m)=2a_m\Bigl(\frac{1}{s_m}+\eps_m\Bigr).
\end{align*}
Thus~\eqref{Hr7} yields
$$
  \frac{\abs{\psi(z)-\psi(q_m)}}{\abs{\psi(z)-\psi(x)}}
  \geq\frac{a_m(1-\eps_m)}{2a_m\bigl(\frac{1}{s_m}+\eps_m\bigr)}=
  \frac{1-\eps_m}{2\bigl(\frac{1}{s_m}+\eps_m\bigr)}>c
$$
and~\eqref{Hr8} follows.
\end{proof}
The next goal is to show that with a proper choice of $\seq{a_n}$ and $\seq{b_n}$
the graph of $f$ is porous. To that end
we introduce the following system of rectangles.
Let $\mc R$ denote the family of all planar rectangles $I\times J$,
where $I,J$ are compact intervals,
with aspect ratio $5:3$, i.e.~$\frac{\leb(I)}{\leb(J)}=\frac53$.
Each $R\in\mc R$ is covered in a natural way by $15$ non-overlapping
closed squares with side one fifth of the length of the base of $R$.
The family of these squares will be denoted $\mc S(R)$.
These squares determine in a natural way five closed columns and three
closed rows.

Given $R\in\mc R$, the length of the base of $R$ is denoted $\ell(R)$.
\begin{lem}
There are sequences $\seq{a_n}$ and $\seq{b_n}$ satisfying hypotheses of
Lemma~\ref{HrB} such that for each $R\in\mc R$ there is a square $S\in\mc S(R)$
such that $\inter S\cap\gr=\emptyset$.
\end{lem}
\begin{proof}
We build the sequences recursively. Let $g_n=\sum_{i\leq n}f_i$, $n\in\Nset$, be the
partial sums of $f$; graphs of $g_n$ are denoted $\ggr_n$.
Our goal is to find $a_n$'s and $b_n$'s so that for each $n$ the following holds:
\begin{equation}\label{HrSq}
  \text{For each $R\in\mc R$ there is a square
  $S\in\mc S(R)$ disjoint with $\ggr_n$.} \tag{$\mathsf{C}_n$}
\end{equation}

Choose $a_0$ and $b_0$ so that $s_0>3$. The graph of $f_0$
is obviously covered by three lines: two skewed and one horizontal.
Let $R\in\mc R$.
Each of the two skewed lines, because of their big slopes,
can meet at most two out of the five columns. Therefore one column remains left.
The horizontal line meets at worst two of the three squares forming this column.
Thus one square remains disjoint with each of the three lines and thus
with the graph $\ggr_0$ of $g_0=f_0$. Thus condition $\mathsf C_0$ is met.

Now suppose that $a_i$ and $b_i$ are set up for all $i<n$ so that
condition $\mathsf C_{n-1}$ is met. Let
$\eps_n=\min\{\abs{q_i-q_j}:0\leq i<j\leq n\}$.
\begin{claim}
There is $\delta_n>0$ such that if $\ell(R)\geq\eps_n$,
then there is $S\in\mc S(R)$ that is at least $\delta_n$ far apart from
$\ggr_{n-1}$.
\end{claim}
\begin{proof}
Suppose the contrary: For each $m$ there is $R_m\in\mc R$ such that
$\ell(R_m)\geq\eps_n$ and the distance of $S$ from $\ggr_{n-1}$ is less than
$\frac1m$ for each square $S\in\mc S(R_m)$. In particular, $\ell(R_m)\leq5$ for all
$m\geq1$ and there is a bounded set that
contains all rectangles $R_m$. Thus passing to a subsequence
we may suppose that $\seq{R_m}$ is convergent in the Hausdorff metric.
The limit $R$ of this sequence is clearly a rectangle with aspect ratio $5:3$
or a point.
But the latter cannot happen, because $\ell(R_m)\geq\eps_n$ for each $m$.
Thus $R\in\mc R$. The distance of $\ggr_{n-1}$ from
each of the squares $S\in\mc S(R)$ is obviously zero. Since the squares are
compact,
$\ggr_{n-1}$ meets all of them: the desired contradiction.
\end{proof}
Choose $a_n<\delta_n$ and $b_n$ subject to
\begin{equation}\label{Hr11}
  a_n\leq\frac{2^{-n}}{n}, \qquad s_n>2^n\sum_{i<n}s_i.
\end{equation}
We need to show that thus chosen values ensure condition $\mathsf C_n$.

Suppose first that $\ell(R)\geq\eps_n$. There is $S\in\mc S(R)$ such that
$\dist(S,\ggr_{n-1})\geq\delta_n$. Consequently
\begin{align*}
  \dist(S,\ggr_n)&\geq\dist(S,\ggr_{n-1})-\dist(\ggr_{n-1},\ggr_n)\\
  &\geq
  \delta_n-\max\abs{g_{n-1}-g_n}=\delta_n-\max\abs{f_n}=\delta_n-a_n>0.
\end{align*}
Thus $S$ is disjoint with $\ggr_n$.

To treat the case $\ell(R)<\eps_n$ we first prove
\begin{claim}
If $g_n(x)>0$ and a local maximum of $g_n$ occurs at $x$, then
$x=q_j$ for some $j\leq n$.
\end{claim}
\begin{proof}
Suppose $g_n(x)>0$ and there is a local maximum of $g_n$ at $x$.
We examine the left-sided derivative $g_n^-(x)$. Clearly
$g_n^-(x)=\sum_{i\leq n}f_i^-(x)$ and each $f_i^-(x)$ is either $0$, or
$s_i$, or $-s_i$.
If all of them were $0$, the value $g_n(x)$ would be $0$, so there is
$i\leq n$ such that
$f_i^-(x)\neq0$. Let $j=\max\{i\leq n:f_i^-(x)\neq0\}$.
First of the conditions~\eqref{Hr11} yields $\abs{\sum_{i<j}f_i^-(x)}<s_j$.
Since $g_n^-(x)\geq0$, it follows that $f_j^-(x)=s_j$.

By the same analysis of the right-sided derivative, letting
$k=\max\{i\leq n:f_i^+(x)\neq0\}$ we have $f_k^+(x)=-s_k$.

Suppose that $j<k$. Then, by the definition of $j$, $f_k^-(x)=0$ and
$f_k^+(x)=-s_k$. But there is no such point. Thus $j<k$ fails.
The same argument proves that $j>k$ fails as well. Therefore $j=k$.
Overall, $f_j^-(x)=s_j$ and $f_j^+(x)=-s_j$. The only point with this property
is $q_j$.
\end{proof}
Now suppose $R=I\times J\in\mc R$ and that $\ell(R)<\eps_n$.
It is clear that if the graph $\ggr_n$ passes through all squares $S\in\mc S(R)$,
then $g_n$ has at least two positive local maxima in $I$. Therefore, by the above Claim,
there are $i<j\leq n$ such that both $q_i$ and $q_j$ belong to $I$.
Consequently $\abs{q_i-q_j}\leq\leb(I)=\ell(R)<\eps_n$,
which contradicts the definition of $\eps_n$. Thus
$\ggr_n$ misses at least one of the squares $S\in\mc S(R)$.
The proof of condition $\mathsf C_n$ is finished.

It remains to draw the statement of the lemma from conditions $\mathsf C_n$.
Fix $R\in\mc R$. Since there are only finitely many squares in $\mc S(R)$,
there is $S\in\mc S(R)$ such that the set $F=\{n:\ggr_n\cap S=\emptyset\}$
is infinite. Since $f=\lim_{n\in F}g_n$, we have
$\gr\subs\cl{\bigcup_{n\in F}\ggr_n}$.
Therefore $\gr$ does not meet $\inter S$.

Conditions~\eqref{Hr11} ensure that $f$ satisfies hypotheses
of Lemma~\ref{HrB}.
\end{proof}
\subsection*{Proof of Theorem~\ref{Hrrrr}}
The required function $f$ is of course the one constructed in the above lemma.
Let $B(x,r)$ be any closed ball in $\Rset^2$. Inscribe in $B(x,r)$ a rectangle
$R\in\mc R$, as big as possible. By the above lemma there is a square
$S\in\mc S(r)$ such that $\inter S$ misses $\gr$. Inscribe into $S$ an open ball $B$.
This ball is disjoint with $\gr$. The radius of this ball is by trivial calculation
$r/\sqrt{34}$. The closed ball concentric with $B$ and of radius $\frac r6$
is thus disjoint
with $\gr$. We proved that $\gr$ is strongly porous.

The function $f$ is absolutely
continuous by Lemma~\ref{HrA} and $\gr$ is not \si monotone by Lemma~\ref{HrB}.
\qed

\medskip
Since any monotone function has trivially a $1$-monotone graph, and since
every absolutely continuous function is a difference of two increasing
functions, we have
\begin{coro}
A sum of two functions with $1$-monotone graphs need not have a
\si monotone graph.
\end{coro}

%%%%%%%%%%%%%%%%%%%%%%%%%%%%%%%%%%%%%%%%%%%%%%%%%%%%%%%%%%%%%%%%%%%%%%%%%%%%%%%%
%%%%%%%%%%%%%%%%%%%%%%%%%%%%%%%%%%%%%%%%%%%%%%%%%%%%%%%%%%%%%%%%%%%%%%%%%%%%%%%%
\section{Remarks and questions}\label{sec:Q}
%%%%%%%%%%%%%%%%%%%%%%%%%%%%%%%%%%%%%%%%%%%%%%%%%%%%%%%%%%%%%%%%%%%%%%%%%%%%%%%%
%%%%%%%%%%%%%%%%%%%%%%%%%%%%%%%%%%%%%%%%%%%%%%%%%%%%%%%%%%%%%%%%%%%%%%%%%%%%%%%%

We conclude with several remarks and questions that we consider interesting.

\subsection*{Hausdorff dimension}
If a continuous function $f:I\to\Rset$ has a monotone graph, then
$\hdim\gr=1$  by Corollary~\ref{coroHaus}.
The analogy for \si monotone graph fails:
\begin{prop}\label{chdim}
There is a continuous function $f:[0,1]\to\Rset$ with a \si monotone graph
such that $\hdim\gr>1$. Any such function admits a perfect set
of non-$\MM$-points.
\end{prop}
\begin{proof}
There is a continuous function $g:[0,1]\to\Rset$ such that $\hdim\ggr>1$,
cf.~e.g.~\cite[Chapter 11]{MR1102677}. By~\cite{KMZ}, there is a monotone
compact set $K\subs\ggr$ such that $\hdim K>1$.
Let $C=\{x\in[0,1]:\psi(x)\in K\}$.
Define $f$ to coincide with $g$ on $C$ and on each component of
the complement of $C$ let $f$ be linear and so that it is continuous on $[0,1]$.
Since there are only countably many components, the resulting function
has a \si monotone graph.

To prove the second statement notice that Theorem~\ref{thmHaus} yields
$\hdim\gr|\MM(f)=1$ and thus if $\hdim\gr>1$, then the set of non-$\MM$-points
certainly contains a perfect set.
\end{proof}

\subsection*{Nowhere differentiable functions}
The nowhere differentiable function of Proposition~\ref{function} has no
$\MM$-points. Though there is a plethora of other nowhere differentiable functions
with the same property and the argument for nonexistence of $\MM$-points
seems similar to that for nonexistence of derivatives,
in general we know about nowhere differentiable functions only
Corollary~\ref{coroMM2}: the set of $\MM$-points is meager.
\begin{question}
Is there a continuous nowhere differentiable function with a dense or even
perfectly dense set of $\MM$-points? What about $\MM_1$-points?
\end{question}

The Baire category arguments used cannot be adapted to subsets of a graph that
are of positive measure, since such sets may be totally disconnected and thus
have way too many candidates for witnessing order to check.
\begin{question}
Let $f$ be the function of Proposition~\ref{function}.
Is there a set $A\subs[0,1]$ of positive measure such that $\gr|A$
is monotone?
\end{question}

\subsection*{Bounded variation}
By Theorem~\ref{1-mono}, a continuous function with a $1$-monotone graph
is of bounded variation.
It also follows from Proposition~\ref{approximate} that a continuous function
with a monotone, rectifiable graph is differentiable almost everywhere.
\begin{question}
Is there a continuous function on $[0,1]$ with a monotone, rectifiable graph
that is not of bounded variation?
\end{question}

\subsection*{Luzin property}
Recall that $f$ satisfies \emph{Luzin condition} if $\leb(f(A))=0$ whenever $\leb(A)=0$.
Note that if $f$ has a monotone graph, then it satisfies Luzin condition
``almost everywhere'':
Letting $\DD_\infty=\{x\in\DD(f):\abs{f'(x)}=\infty\}$,
we have $\leb(\DD_\infty)=0$ and if $A\cap\DD_\infty=\emptyset$, then
$\leb(A)=0$ implies $\leb(f(A))=0$. Hence $f$ satisfies Luzin condition if and only if
$\leb(f(\DD_\infty))=0$.

The following easily follows from Theorem~\ref{bigder}.
\begin{prop}
A continuous function satisfying Luzin condition with a \si mono\-tone graph
is differentiable
at a set that has positive measure within each interval.
\end{prop}
\begin{question}
Is a continuous function satisfying Luzin condition with a monotone graph differentiable
almost everywhere?
\end{question}

\subsection*{Porosity constant}
We know from~\cite[Theorem 4.2]{HrZi} that any monotone set in $\Rset^2$ is
strongly porous, and from Theorem~\ref{Hrrrr} that the converse fails.
In our proof we showed that the porosity
constant of $\gr$ can be pushed to $1/\sqrt{34}$.
Perhaps a set must be \si monotone if it is strongly porous and its
porosity constant is large enough? For compact sets in the plane it is not so:
Let $C\subs\Rset$ be a strongly porous perfect set such that every
$p<\frac12$ is a porosity constant of $C$. The set $C\times[0,1]$ clearly
has the same property. On the other hand, by~\cite[Lemma 2.1]{HrZi}
it is not \si monotone.
Hence there is a strongly porous compact set $X\subs\Rset^2$  such that every
$p<\frac12$ is its porosity constant and yet $X$ is not \si monotone.
But what about curves and graphs?
\begin{question}
Is there $p<\frac12$ such that every strongly porous curve in $\Rset^2$ with porosity
constant $p$ is monotone or \si monotone?
What about graphs of continuous functions?
\end{question}

\subsection*{Monotone graph vs.~continuity}
Say that a function $f:I\to\Rset$ is \emph{\si continuous} if there is a partition
$\{D_n:n\in\Nset\}$ of $I$ such that $f\rest D_n$ is continuous for each $n$.
We claim that monotone graph does not imply \si continuity.
To see that, let $C\subs[0,1]$ be the usual Cantor ternary set and $g:C\to C$
a non-\si continuous function. By~\cite[Proposition 4.6]{HrZi} $C\times C$ is
monotone. Therefore so is the graph of $g$. Now extend $g$
to $f:[0,1]\to\Rset$ by $f(x)=-1$ for $x\notin C$. Easy to check that $\gr$
is monotone and yet not \si continuous.

How about $1$-monotone graphs? Consider the function $f:[0,1]\to\Rset$ defined
by $f(x)=-1$ if $x$ is rational and $f(x)=x$ otherwise. The graph of $f$ is
$1$-monotone, but $f$ is continuous at no point. However, $f$ is continuous on
both rationals and irrationals.
\begin{question}
Is a function with a $1$-monotone graph \si continuous?
\end{question}

%
%Lower porosity is usually defined as follows: A set $X\subs\Rset^2$ is
%\emph{lower porous} if
%there is $p>0$ (called \emph{lower porosity constant} such that for any $x\in X$
%there is $r_0>0$ such that for any $r\in(0,r_0)$
%there is $y\in\Rset^2$ such that $B(y,pr)\subs B(x,r)\setminus X$.
%Analysis of our proof of Theorem~\ref{Hrrrr} reveals that with this notion
%of porosity everything works for 4-by-3 rectangles and thus the lower porosity
%constant can be pushed to $1/5$.

% BIBLIOGRAPHY ----------------------------------------------------
\bibliographystyle{amsplain}
%\bibliography{nwdiff}
\providecommand{\bysame}{\leavevmode\hbox to3em{\hrulefill}\thinspace}
\providecommand{\MR}{\relax\ifhmode\unskip\space\fi MR }
% \MRhref is called by the amsart/book/proc definition of \MR.
\providecommand{\MRhref}[2]{%
  \href{http://www.ams.org/mathscinet-getitem?mr=#1}{#2}
}
\providecommand{\href}[2]{#2}

% -----------------------------------------------------------------
\end{document}